\RequirePackage[T1]{fontenc}
\documentclass[12pt,twoside]{amsart}
\calclayout

\usepackage[whole]{bxcjkjatype}
\usepackage[utf8]{inputenc}
\usepackage{amsmath}
\usepackage{amssymb}
\usepackage{mathtools}
\numberwithin{equation}{section}
\usepackage{slashed}
\usepackage{braket}
\usepackage[svgnames]{xcolor}
\usepackage{colortbl}
\definecolor{lightgray}{gray}{0.9}
\definecolor{xgray}{gray}{0.8}
\colorlet{conjcolor}{blue!5!white}
\usepackage[colorlinks,citecolor=DarkGreen,linkcolor=FireBrick,urlcolor=FireBrick,linktocpage,breaklinks=true]{hyperref}
\usepackage{cite}
\usepackage[pdftex]{graphicx}
\usepackage{times}
\usepackage{courier}
\usepackage{bm}
\usepackage{subfig}
\usepackage{fnpct}

\usepackage{mathrsfs} 

\usepackage[all]{xy}
\usepackage{enumitem}

\theoremstyle{plain}
\newtheorem{defn}[equation]{Definition}

\newtheorem{thm}[equation]{Theorem}

\newtheorem{prop}[equation]{Proposition}

\newtheorem{fact}[equation]{Fact}
\newtheorem{const}[equation]{Construction}
\def\PhysicsFact{Physics Assumption}

\newtheorem{fact?}[equation]{Fact?}

\newtheorem{lem}[equation]{Lemma}
\newtheorem{ass}[equation]{Assumption}

\newtheorem{claim}[equation]{Claim}
\theoremstyle{remark}
\newtheorem{rem}[equation]{Remark}
\newtheorem{ex}[equation]{Example}

\usepackage[many]{tcolorbox}

\tcbset{
  breakable,
  colback=conjcolor,
  boxrule=-1pt,
  boxsep=1pt,
  left=2pt,right=2pt,top=2pt,bottom=2pt,
  oversize=2pt,
  sharp corners,
  before skip=\topsep,
  after skip=\topsep,
}
\tcolorboxenvironment{proofc}{}
\tcolorboxenvironment{defnc}{}
\tcolorboxenvironment{propc}{}
\tcolorboxenvironment{thmc}{}
\tcolorboxenvironment{corc}{}

\let\oldendrem\endrem
\def\endrem{\hfill {$\lrcorner$} \oldendrem}

\def\bC{\mathbb{C}}

\def\bZ{\mathbb{Z}}
\def\bS{\mathbb{bS}}

\def\Hom{\mathrm{Hom}}
\def\Ext{\mathrm{Ext}}

\def\K{\mathrm{KU}}
\def\KO{\mathrm{KO}}
\def\KU{\mathrm{KU}}

\def\MF{\mathrm{MF}}
\def\TMF{\mathrm{TMF}}

\def\tmf{\mathrm{tmf}}
\def\mf{\mathrm{mf}}

\def\Z{\mathbb{Z}}
\def\MString{\mathrm{MString}}
\def\MSpin{\mathrm{MSpin}}

\def\Q{\mathbb{Q}}

\def\Nequals#1{$\mathcal{N}{=}#1$}

\def\spin{\text{spin}}
\def\stri{\text{string}}

\def\Wit{\mathop{\mathrm{Wit}}\nolimits}

\def\SQFT{\mathrm{SQFT}}

\def\del{\partial}

\def\id{\mathrm{id}}

\def\fr{\mathrm{fr}}

\def\rel{\mathrm{rel}}
\def\SQM{\mathrm{SQM}}
\def\AHR{\mathrm{AHR}}
\def\ABS{\mathrm{ABS}}
\def\Ind{\mathrm{Ind}}
\def\Ch{\mathrm{Ch}}

\def\Lie{\mathrm{Lie}}
\def\Ph{\mathrm{Ph}}

\let\oldtext\text
\def\text#1{\oldtext{\upshape\mdseries #1}}

\DeclareSymbolFont{yhlargesymbols}{OMX}{yhex}{m}{n} \DeclareMathAccent{\reallywidehat}{\mathord}{yhlargesymbols}{"62}

\DeclareFontFamily{U}{mathx}{}
\DeclareFontShape{U}{mathx}{m}{n}{<-> mathx10}{}
\DeclareSymbolFont{mathx}{U}{mathx}{m}{n}
\DeclareMathAccent{\widehat}{0}{mathx}{"70}
\DeclareMathAccent{\widecheck}{0}{mathx}{"71}

\let\oldwidehat\widehat
\def\widehat#1{\oldwidehat{#1}{}}

\def\paragraph#1{

\medskip\noindent \textit{#1} --- }

\begin{document}

\title[576 periodicity]{On the $576$-fold periodicity of the spectrum SQFT: The proof of the lower bound via the Anderson duality pairing}
\author{Theo Johnson-Freyd}
\author{Mayuko Yamashita}
\date{}
\address{Department of Mathematics and Statistics, Dalhousie University, and Perimeter Institute for Theoretical Physics}
\email{theojf@pitp.ca}
\address{Department of Mathematics, Kyoto University, 
Kita-shirakawa Oiwake-cho, Sakyo-ku, Kyoto, 606-8502, Japan
}
\email{yamashita.mayuko.2n@kyoto-u.ac.jp}

\thanks{The authors thank Sanath Devalapurkar for suggesting the Toda bracket construction of manifolds appearing in the proof of the main theorem. 
The work of TJF is supported by the NSERC grant RGPIN-2021-02424 and by the Simons Foundation grant 888996.
 The work of MY is supported by Grant-in-Aid for JSPS KAKENHI Grant Number 24K00522, 20K14307
and JST CREST program JPMJCR18T6. We both furthermore acknowledge the Simons Collaboration on Global Categorical Symmetries. This project began during a visit by the second-named author to the Perimeter Institute for Theoretical Physics, whom we thank for their hospitality. Research at the Perimeter Institute is supported in part by the Government of Canada through the Department of Innovation, Science, and Economic Developmnet, and by the Province of Ontario through the Ministry of Colleges and Universities.}

\begin{abstract}
We are aimed at giving a differential geometric, and accordingly physical, explanation of the $576$-periodicity of $\TMF$. 
In this paper, we settle the problem of giving the lower bound $576$. 
We formulate the problem as follows: we assume a spectrum $\SQFT$ with some conditions,
suggest from physical considerations about the classifying spectrum for two-dimensional \Nequals{(0, 1)}-supersymmetric quantum field theories,
and show that the periodicity of $\SQFT$ is no less than $576$. The main tool for the proof is the analogue of the Anderson duality pairing introduced by the second-named author and Tachikawa. We do not rely on the Segal-Stolz-Teichner conjecture, so in particular we do not use any comparison map with $\TMF$. 
%
%
%
\end{abstract}
\maketitle

\tableofcontents


\section{Introduction}\label{sec_intro}

There are geometrically and physically important spectra which have interesting periodicities. 
The most classical examples are $\K$ and $\KO$ which are $2$- and $8$-periodic, respectively. There are many interpretations of this periodicity, including geometric and physical ones. 
Another important source of such examples are elliptic spectra. The universal version $\TMF$, the Topological Modular Forms, is exactly $576$-periodic. $\TMF$ is connected to differential geometry via the string orientation \cite{AHR10}, and also of physical importance by the Segal-Stolz-Teichner conjecture \cite{StolzTeichner1} \cite{StolzTeichner2} asserting that $\TMF$ classifies two-dimensional \Nequals{(0,1)}-supersymmetric unitary quantum field theories (SQFTs). 
However, a geometric or physical explanation for the $576$-periodicity of $\TMF$ has been lacking. 
The only existing proof of periodicity is by the spectral sequence computation of the homotopy groups of $\TMF$. 
Our goal is to give a geometric understanding of this $576$-periodicity. In this paper we settle the problem giving the lower bound $576$. Giving the upper bound $576$ is a completely independent problem, and will be the subject of a forthcoming paper.
%

Specifically, we will show that for any $E_\infty$ ring spectrum satisfying a few of the same general structural properties as $\TMF$, the periodicity (possibly infinite) must be divisible by $576$. Our reason for organizing the result this way comes from the aforementioned Segal--Stolz--Teichner conjecture. That conjecture can be broken into several sub-conjectures:
\begin{enumerate}
  \item For a suitable (but as yet unknown) definition of ``1+1D quantum field theory,'' the compact unitary 1+1D minimally-supersymmetric quantum field theories are cocycles for a certain $E_\infty$ ring spectrum $\SQFT$.\label{st1}
  \item Every 1+1D minimally-supersymmetric quantum field theory has a well-defined \emph{topological Witten genus} lifting the modular-form valued Witten genus, and assembling into an $E_\infty$ map $\SQFT \to \TMF$.\label{st2}
  \item This map $\SQFT \to \TMF$ is an isomorphism.\label{st3}
\end{enumerate}
Our motivation for this paper is the belief that at least piece (\ref{st1}) is true, whereas we are not aware of convincing  physical arguments supporting (\ref{st2}) or (\ref{st3}). 
As such, we will assume only that we have some $E_\infty$ ring spectrum $\SQFT$ satisfying the following assumptions:
   \begin{itemize}
        \item (Assumption \ref{ass_string_orientation}, see there for the notations) $\SQFT$ fits into the following commutative diagram of $E_\infty$ spectra, 
     \begin{align}\label{diag_rationalization_SQFT_intro}
        \xymatrix{
        \MString \ar[d]^-{\iota} \ar[rr]^-{\sigma_\stri} \ar@/^18pt/[rrr]^-{\Wit_{\stri}}  && \SQFT \ar[d]^-{\varphi} \ar[r]^-{\Phi} &  H\MF^\Q \ar@{_{(}-_>}[d] \\
        \MSpin \ar[rr]^-{\AHR_{\spin}} \ar@/_18pt/[rrr]_-{\Wit_{\spin}} && \KO((q)) \ar[r]^-{\mathrm{Ph}} & \Sigma^{4\Z}H\Q((q))
        }. 
    \end{align}
    \item (Assumption \ref{ass_21}) We have
    \begin{align}\label{eq_ass_21_intro}
        \pi_{-21}\SQFT = 0. 
    \end{align}
    \end{itemize}
We note that $\TMF$ does satisfy these assumptions. 

A mathematician uninterested in the Segal--Stolz--Teichner conjecture can simply take these assumptions as the starting point for an interesting puzzle: how much about $\TMF$ is already implied by these assumtions? But a physicist should ask if they are motivated. We believe that the first assumption, saying that the various versions of the Witten genus factor through through $\SQFT$, is physically unimpeachable: indeed, Witten discovered his genus by thinking about supersymmetric sigma models. The second assumption is mathematically clean but not physically motivated. In fact, we will use it only to prove a mathematically more complicated but physically more natural statement. Following \cite{Tachikawayamashita2021,tachikawa2023anderson}, we will (\S\ref{subsec_pairing}) produce a certain map $\alpha_\stri: \SQFT \to \Sigma^{-20}I_\bZ \MString$, whose physical significance is to compute the anomaly of heterotic string compactification. The first assumption on $\SQFT$ already implies that $\alpha_\stri$ vanishes rationally, and the second assumption then lifts this vanishing result to an integral statement, and it is this vanishing that we will use to give our lower bound on periodicity. As such, a physicist could decide to replace our second assumption with the physically well-motivated (but physically unproven) expectation:
\begin{itemize}
    \item The map $\alpha_\stri : \SQFT \to \Sigma^{-20}I_\bZ \MString$ encoding the heterotic string anomaly is trivial.
\end{itemize}

Under these assumptions, we will show our main theorem:
\begin{thm}[{= Theorem \ref{thm_main}}]\label{thm_intro}
    There are no invertible elements in $\pi_n \SQFT$ for $n$ not divisible by $576$. 
    In other words, the periodicity (possibly $\infty$) of the spectrum $\SQFT$ is at least $576$. 
\end{thm}

First observe that the upper right horizontal arrow of \eqref{diag_rationalization_SQFT_intro} induces the ring homomorphism
\begin{align}\label{eq_Phi_intro}
    \Phi \colon \pi_{*}\SQFT \to \MF_{*/2} = \Z[c_4, c_6, \Delta, \Delta^{-1}]. 
\end{align}
Here we can take the codomain to be the ring of {\it integral} weakly holomorphic modular forms, since the right square of \eqref{diag_rationalization_SQFT_intro} is commutative. 
Theorem \ref{thm_main} follows from the following Proposition. 
\begin{prop}[{= Proposition \ref{prop_main}}]\label{prop_intro}
    The periodicity (possibly $\infty$) of the graded subring
    \begin{align}\label{eq_prop_intro}
        \mathrm{Im}\left( \Phi \colon \pi_{*}\SQFT \to \MF_{*/2}\right) \subset \MF_{*/2}
    \end{align}
    is at least $576$. 
\end{prop}
We note that, since the invertible elements in $\MF_{*/2}$ are precisely the ones of the form $\pm\Delta^d$ for $d \in \Z$, the periodicity of the graded ring \eqref{eq_prop_intro} should be a multiple of $24$. 


Our strategy for the proof of Proposition \ref{prop_intro} is to use the {\it Anderson duality pairing}. 
The two assumptions for $\SQFT$ above allows us to construct an interesting map of spectra
\begin{align}
        \alpha_{\spin/\stri} \colon \SQFT \to \Sigma^{-20}I_\Z \MSpin/\MString, 
\end{align}
which induces a pairing (Construction \ref{const_pairing_general})
\begin{equation}\label{eq_pairing_intro}
\langle-,-\rangle_\SQFT \colon \pi_{-4k}\SQFT \times \pi_{4k-20}\MSpin/\MString\to \bZ,
\end{equation}
with the following explicit formula, 
\begin{align}\label{eq_formula_pairing_intro}
    \langle [\mathcal{T}], [M, N] \rangle_{\alpha_{\spin/\stri}} 
        =\left.\frac{1}{2}\Delta \cdot \Phi([\mathcal{T}]) \cdot \Wit_\rel([M, N]) \right|_{q^0}
\end{align}
Here $\Wit$ is the relative Witten genus for spin manifods with string boundary. 
The crucial point is the well-definedness and the integrality of the pairing. 
As we explain below, this duality pairing allows us to show {\it non-existence} results by {\it existence} results of manifolds with specific invariants. 

Let us illustrate our method by showing that periodicity of the graded ring \eqref{eq_prop_intro} in question is strictly bigger than $24$. For that, it is enough to obstruct the existence of $\Delta^{-1}$ in \eqref{eq_prop_intro}. 
Actually, we show the following stronger claim. 
\begin{prop}\label{prop_1_intro}
    If $k \Delta^{-1}$ is contained in the image of $\Phi\colon \pi_{- 24}\SQFT \to \MF_{-12}$, 
    then we have $24|k$. 
\end{prop}
\begin{proof}
    We use the pairing for $4k = 24$. 
We know that $\pi_4\MSpin/\MString \simeq \Z$, and it is generated by $[D^4, S^3_\Lie]$, the $4$-disk with boundary $S^3$ with the $SU(2)$ Lie group framing. 
We can easily verify that $\Wit_\rel([D^4, S^3_\Lie]) = \frac{1}{12} + O(q)$ (Lemma \ref{lem_Ind_rel_D4S3}). 
Thus, if we have $k\Delta^{-1} \in \mathrm{Im}(\Phi \colon \pi_{-24}\SQFT \to \MF_{-12})$, we have
\begin{align}
    \Z \ni\left. \frac{1}{2} \Delta \cdot k\Delta^{-1} \cdot \left(\frac{1}{12} + O(q)\right)\right|_{q^0} = \frac{k}{24}. 
\end{align}
Thus $k$ should be divisible by $24$. 
\end{proof}

We prove Proposition \ref{prop_intro} the case for other $d$ by the analogous ways, using the pairing for corresponding degrees. 
For that, we should produce interesting elements in $\pi_{-24d-20} \MSpin/\MString$ with specific relative Witten genera. This is another interesting point of this work. 
We do this by referring to the information on Adams-Novikov spectral sequences for $\tmf$ and related spectra, in particular of the information of the Toda brackets. 
We produce manifolds by a manifold-level construction of Toda brackets, explained in Appendix \ref{app:Toda}. 
The authors believe that this construction should be of independent interest. 

The use of the Anderson duals in this context comes from its significance in the study of anomalies in QFTs \cite{FreedHopkins2021}. 
In particular this work is motivated from the previous works of MY with Y.~Tachikawa \cite{Tachikawayamashita2021} and \cite{tachikawa2023anderson}, where they consider the corresponding Anderson duality pairing, where instead of $\SQFT$ above we used $\TMF$. 
Our key observation here is that their construction only uses the fact that the spectrum $\TMF$ satisfies the assumptions listed above. 
Therefore, in this paper we define the pairing \eqref{eq_pairing_intro} by just replacing ``$\TMF$'' in \cite{tachikawa2023anderson} to ``$\SQFT$''.

This paper is organized as follows. We begin in Section~\ref{sec:warmup} with a warm-up: using the same methods as for our main theorem, we will show that an $E_\infty$ ring spectrum ``$\SQM$'' satisfying some of the same properties as $\KO$ must have periodicity at least $8$. The choice of name ``$\SQM$'' comes from the 0+1D version of the Segal--Stolz--Teichner conjecture: there is a cocycle model for $\KO$ whose cocycles are minimally supersymmetric quantum mechanics models with $T^2 = +1$ time-reversal symmetry \cite{HohnholdStolzTeichner}. The proof of our main theorem occupies Section~\ref{sec_lowerbound}. Finally, a short Appendix~\ref{app:Toda} recalls some details about Toda brackets that will be needed in our proof.

\subsection{Notations and conventions}

\begin{itemize}
\item We basically follow the standard notations of elements in $p$-local stable homotopy groups for $p=2, 3$, which can be found in e.g., \cite{Ranavel86}. 
By abuse of notations, we use the same notations for the Hurewicz images to $E_\infty$ ring spectra. 
    \item We use the notation $\Sigma^{2\Z} H\Q := \vee_{k \in \Z} \Sigma^{2k} H\Q $. $\Sigma^{4\Z}H\Q$ and $\Sigma^{4\Z+2}H\Q$ are defined similarly. 
    \item We denote by $S^1_\Lie$ and $S^3_\Lie$ the manifold $S^1$ and $S^3$ equipped with the Lie group framing by $S^1 \simeq U(1)$ and $S^3 \simeq SU(2)$, respectively. 
    We have $[S^1_\Lie]=\eta \in \pi_1 \bS$, $[S^3_\Lie] = \nu \in \pi_3 \bS_{(2)}$ and $[S^3_\Lie] = \alpha_1 \in \pi_3 \bS_{(3)}$. 

    \item We denote the complex Bott element by $\beta \in \pi_2 \K$. When we identify $\pi_{2n} \KU \simeq \Z$, we always take the sign so that $\beta^n$ maps to $1$. 
    We identify $\pi_{8k+4}\KO \simeq 2\Z \subset \Z$ and $\pi_{8k}\KO \simeq \Z$ via $\pi_{4n}\KO \xhookrightarrow{c}\pi_{4n}\KU \simeq \Z$, where $c$ denotes the complexification. 

    \item We denote by $\Ch \colon \K \to \Sigma^{2\Z}H\Q$ and $\Ph \colon \KO \to \Sigma^{4\Z}H\Q$ the Chern character and the Pontryagin character respectively. They fit into the commutative diagram
    \begin{align}
    \xymatrix{
    \KU \ar[r] \ar[d]^-{c} & \Sigma^{2\Z}H\Q \ar@{_{(}-_>}[d] \\
    \KO \ar[r] & \Sigma^{4\Z}H\Q
    }
    \end{align}
    
    \item $\MF_* := \Z[c_4, c_6, \Delta, \Delta^{-1}]/(c_4^3-c_6^2 - 1728\Delta)$ denotes the $\Z$-graded ring of {\it weakly holomorphic} modular forms. We have $\deg c_4 = 4$, $\deg c_6 = 6$ and $\deg \Delta = 12$. 
    We have $\MF_k = 0$ for odd $k$, and
    we use the standard inclusion $\MF_k \hookrightarrow \Z((q))$ for each even $k$. This induces the inclusion
    \begin{align}\label{eq_MF_notation}
        \MF_* \hookrightarrow \oplus_{k \in \Z} \Z((q))[2k]
    \end{align}
    of $\Z$-graded rings. Here $[k]$ denotes the degree shift by $k$.  
    Let us denote $\MF_*^\Q := \MF_* \otimes\Q$.

\end{itemize}

\section{Warm-up: The $\ge 8$-periodicity of the spectrum $\SQM$} \label{sec:warmup}
This section is a warm-up for the next section, which is the main part of this paper. 
We consider the baby version of the main problem,
which is to give a geometric understanding of the lower bound $8$ of the periodicity for a class of spectra which includes $\KO$-theory. 
This section is logically independent from the rest of the paper, but we recommend the reader to start from this section because it illustrates our basic strategy in much simpler ways. 

We formulate the problem in the same way as that for $\SQFT$ outlined in the Introduction. 
In this section, we do NOT use any knowledge of $\KO$-theory. 
But we allow ourselves to use the knowledge of $\MSpin$, $\MSpin^c$ and $\K$, as well as the Atiyah-Bott-Shapiro orientation for $\K$-theory. 
We set up the problem in Subection \ref{subsec_assumptions_SQM}. 
We start from an arbitrary $E_\infty$ ring spectrum $\SQM$, and put assumptions which connects it with spin manifolds and $\K$-theory, as well as a vanishing assumption of homotopy group of a particular degree, namely $-3$. 
Then in the rest of this section we show that the periodicity of $\SQM$ is at least $8$ (Theorem \ref{thm_SQM}). 
The strategy is also parallel to the one for $\SQFT$ outlined in the Introduction.  
We construct a $\Z$-valued Anderson duality pairing (Definition \ref{def_SQMpairing}) in Subsection \ref{subsec_paring_SQM}. 
Then we use it to prove Theorem \ref{thm_SQM} in Subsection \ref{subsec_proof_SQM}.

\subsection{Assumptions for the spectrum $\SQM$}\label{subsec_assumptions_SQM}
Here we list the mathematical assumptions for the spectrum $\SQM$. 
\begin{ass}\label{ass_spin_orientation}
    $\SQM$ is an $E_\infty$ ring spectrum equipped with an $E_\infty$ ring maps
    \begin{align}
        \rho_{\spin} &\colon \MSpin \to \SQM, \\
        \psi &\colon \SQM \to \K, \\
        \Psi &\colon \SQM \to \Sigma^{4\Z} H\Q,
    \end{align}
    which fit into the following commutative diagram in the $(\infty, 1)$-category of $E_\infty$ ring spectra, 
     \begin{align}\label{diag_spin_orientation}
        \xymatrix{
        \MSpin \ar[d]^-{\iota} \ar[rr]^-{\rho_\spin} \ar@/^18pt/[rrr]^-{\Ind_{\spin}}  && \SQM \ar[d]^-{\psi} \ar[r]^-{\Psi} &  \Sigma^{4\Z}H\Q \ar@{_{(}-_>}[d] \\
        \MSpin^c \ar[rr]^-{\ABS_{\spin^c}} \ar@/_18pt/[rrr]_-{\Ind_{\spin^c}} && \K \ar[r]^-{\Ch} & \Sigma^{2\Z}H\Q
        }. 
    \end{align}
    where the left vertical arrow is the forgetful map, right vertical arrow is the canonical inclusion and $\ABS_{\spin^c} \colon \MSpin^c \to \K$ is the Atiyah-Bott-Shapiro orientation for $\K$. 
    Here, as written in the diagram, we define $\Ind_\spin := \Psi \circ \rho_\spin$ and $\Ind_{\spin^c} := \Ch \circ \ABS_{\spin^c}$. 
\end{ass}



\begin{rem}
    $\KO$ satisfies Assumption \ref{ass_spin_orientation}, by setting $\rho_\spin := \ABS_{\spin}$, the ABS orientation of spin manifolds, and $\Psi \colon = \Ph$. 
\end{rem}

\begin{rem}
    We use the notation $\Ind_{\spin}$ and $\Ind_{\spin^c}$ because they are the map extracting the $\Z$-valued Fredholm index of the Dirac operators. 
    Indeed, the morphism $\Ind_{\spin^c} := \Ch \circ \ABS_{\spin^c}$ is such a map by the Atiyah-Singer index theorem. 
    For $\Ind_{\spin}$, although we do not define it by ABS orientation for spin manifolds, the commutativity of \eqref{diag_spin_orientation} forces us to conclude that $\Ind_\spin$ coincides with the composition $\MSpin \xrightarrow{\ABS_{\spin}} \KO \xrightarrow{\Ph} \Sigma^{4\Z}H\Q$, so that it is the map extracting $\Z$-valued (Not $\Z/2$-valued) Dirac indices. 
\end{rem}

\begin{ass}\label{ass_SQM_3}
    We have
    \begin{align}
        \pi_{-3}\SQM = 0. 
    \end{align}
\end{ass}

\begin{rem}
    Of course $\KO$ satisfies Assumption \ref{ass_SQM_3}. 
\end{rem}

In this settings, we prove the following theorem. 

\begin{thm}\label{thm_SQM}
    There is no invertible elements in $\pi_n \SQM$ for $n \not\equiv 0 \pmod{8}$. 
    In other words, the periodicity (possibly $\infty$) of the spectrum $\SQM$ is at least $8$. 
\end{thm}

\subsection{Relative Dirac indices}\label{subsec_relative_Ind}
As a preparation for the proof of Theorem \ref{thm_SQM}, we discuss the consequence of Assumption \ref{ass_spin_orientation}. In particular we introduce the {\it relative Dirac indices}, which appear in the formula for the Anderson duality pairings \eqref{eq_formula_SQMpairing}. 

The commutative diagram \eqref{diag_spin_orientation} induces the following commutative diagram where the rows are homotopy fiber sequences, 
\begin{align}\label{diag_K/SQM}
\xymatrix{
\MSpin \ar[rrr]^-{\iota} \ar[d]^-{\rho_\spin}\ar@/_50pt/[ddd]_(.2){\Ind_{\spin}}  &&& \MSpin^c \ar[rrr]^-{C\iota} \ar[d]^-{\ABS_\spin^c} \ar@/_50pt/[ddd]_(.2){\Ind_\spin^c}  &&& \MSpin^c/\MSpin \ar[d]^-{\rho_\rel} \ar@/_70pt/[ddd]_(.2){\Ind_\rel} \\
\SQM \ar[rrr]^(.3){\psi} \ar[d]^-{\otimes \Q}\ar@/_30pt/[dd]_(.2){\Psi} &&& \K \ar[rrr]^(.3){C\psi} \ar[d]^-{\otimes \Q} \ar@/_30pt/[dd]_(.2){\Ch}  &&& \K/\SQM \ar[d]^-{\otimes \Q}\ar@/_40pt/[dd]_(.3){\Ch/\Psi} \\
\SQM_\Q \ar[rrr]^(.3){\psi_\Q} \ar[d]^-{\Psi_\Q} &&& \K_\Q \ar[rrr]^(.3){C\psi_\Q} \ar[d]_-{\simeq}^-{\Ch_\Q} &&& (\K/\SQM)_\Q \ar[d]^-{(\Ch/\Psi)_\Q} \\
\Sigma^{4\Z}H\Q \ar@{_{(}-_>}[rrr] &&& \Sigma^{2\Z}H\Q \ar[rrr]_-{\mathrm{proj}} &&& \Sigma^{4\Z + 2}H\Q. 
}
\end{align}
Here we choose the bottom right horizontal arrow $\Sigma^{2\Z}H\Q \to \Sigma^{4\Z+2}H\Q$ to be the projection so that the induced map on $\pi_{4\Z+2}$ becomes $\id_\Q \colon \Q \to \Q$. 
We define the morphisms $\Ind_\rel$, $\rho_\rel$ and $\Ch/\Psi$ as indicated in the diagram. 
The morphism $\Ind_\rel$ induces the map
\begin{align}\label{eq_rel_Ind}
    \Ind_\rel \colon \pi_{d}\MSpin^c/\MSpin \to \begin{cases} \Q \quad \mbox{if }d \equiv 2 \pmod 4, \\
    0 \quad \mbox{otherwise}
    \end{cases}, 
\end{align}
which we call the {\it relative Dirac index}. 
Let $d \equiv 2 \pmod 4$. By the commutativity of \eqref{diag_K/SQM}, the following diagram commutes. 
\begin{align}\label{diag_Ind_rel_and_Ind}
    \xymatrix{
    \pi_d \MSpin^c \ar[d]_-{C\iota \colon [M]\mapsto [M, \varnothing] } \ar[rrd]^-{\Ind_{\spin^c}} && \\
\pi_d \MSpin^c/\MSpin \ar[rr]_-{\Ind_\rel} && \Q. 
    }
\end{align}

\begin{ex}[{The element $[D^2, S^1_\Lie] \in \pi_{2}\MSpin^c/\MSpin$ and its relative Dirac index}]\label{ex_D2S1}
Let us consider $d=2$ case. We claim that the homotopy fiber exact sequence for bordism groups looks like
\begin{align}\label{diag_MSpinc/MSpin_deg2}
    \xymatrix{
    \pi_2 \MSpin \ar[r]  & \pi_2 \MSpin^c \ar[r] &  \pi_2\MSpin^c/\MSpin \ar[r]^-{\del} & \pi_1 \MSpin \ar[r] & \pi_1 \MSpin^c \\
    \Z/2 \ar[r]^-{0} \ar[u]^-{\simeq}_{\cdot [S^1_\Lie]^2}& \Z \ar[r]^-{\times 2} \ar[u]^-{\simeq}_{\cdot[\mathbb{CP}^1]}& \Z \ar[r]^-{\mod 2}  \ar[u]^-{\simeq}_{\cdot [D^2, {S^1}_\Lie]} & \Z/2 \ar[r] \ar[u]^-{\simeq}_{\cdot [S^1_\Lie]} & 0\ar[u]^-{\simeq}  .
    }
\end{align}
Here, $\pi_i \MSpin$ and $\pi_i \MSpin^c$ are well known, and the fact that $\pi_2 \MSpin^c/\MSpin$ fits into the nontrivial extension between $\pi_2 \MSpin^c \simeq \Z$ and $\pi_1 \MSpin \simeq \Z/2$ is checked by considering the element $[D^2, S^1_\Lie] \in \pi_2\MSpin^c/\MSpin$. 
Here we equip $D^2$ with the trivial spin$^c$-structure, and on the boundary $\del D^2 \simeq S^1$ it is lifted to the nonbounding spin structure. We have
\begin{align}\label{eq_D2S1_generates}
    \del([D^2, S^1_\Lie]) = [S^1_\Lie], \quad \mbox{and} \quad [\mathbb{CP}^1] = 2[D^2, S^1_\Lie]. 
\end{align}
which verifies the above claim. 

The relative Dirac index of the element $[D^2, S^1_{\Lie}]\in \pi_2\MSpin^c/\MSpin$ is computed as follows. 
 \begin{lem}\label{lem_Ind_rel_D2S1}
        We have 
        \begin{align}
            \Ind_{\rel}([D^2, S^1_\Lie]) = \frac{1}{2}. 
        \end{align}
    \end{lem}
    \begin{proof}
        We use the commutativity of \eqref{diag_Ind_rel_and_Ind} and the knowledge of the usual Dirac index for spin$^c$-manifolds. 
        We have
        \begin{align}
            \Ind_\rel([D^2, S^1_\Lie]) = \frac{1}{2} \cdot \Ind_{\rel}(2[D^2, S^1_\Lie]) = \frac{1}{2} \cdot \Ind_{\spin^c}([\mathbb{CP}^1]) = \frac{1}{2}.  
        \end{align}
        Here, for the second equality we used \eqref{eq_D2S1_generates}. 
    \end{proof}
\end{ex}


\subsection{A secondary pairing in the setting of $\SQM$}\label{subsec_paring_SQM}
In this subsection we explain the construction of a morphism (Definition \ref{def_beta_rel})
\begin{align}
    \gamma_{\spin^c/\spin} \colon \SQM \to \Sigma^{-2}I_\Z \MSpin^c/\MSpin
\end{align}
which induces a pairing (Construction \ref{const_pairing_general})
\begin{align}\label{eq_SQMpairing}
    \langle-,-\rangle_{\SQM} \colon \pi_{-4k}\SQM \times \pi_{4k-2}\MSpin/\MString\to \bZ. 
\end{align}
We then show in Proposition \ref{prop_formula_pairing_SQM} that the pairing has the explicit formula
\begin{align}\label{eq_formula_SQMpairing}
    \langle \mathcal{T}, [M, N]\rangle_\SQM
    =\Psi(\mathcal{T}) \cdot \Ind_\rel ([M, N]).  
\end{align}
The nontrivial point is that we get the integrality of the right hand side of \eqref{eq_formula_SQMpairing}.

\begin{defn}\label{def_gamma_K}
    We let $\gamma_{\K} \in \pi_2 I_\Z \K$ to be the Anderson duality element of K-theory shifted by the Bott periodicity.
    In other words, via the isomorphism $\pi_2 I_\Z \K \simeq \Hom(\pi_{-2}K, \Z)\simeq \Hom(\Z\cdot\beta^{-1}, \Z)$, the element $\gamma_{\K}$ corresponds to the isomorphism sending $\beta^{-1}$ to $1$. 
\end{defn}

\begin{prop}\label{prop_vanish_anomaly_SQM}
Let $I_\Z \psi \colon I_\Z \K \to I_\Z \SQM$ denote the Anderson dual to the middle vertical arrow in \eqref{diag_spin_orientation}. 
    We have the following equality in $\pi_{2}I_\Z \SQM$. 
    \begin{align}
        I_\Z \psi (\gamma_\K) = 0. 
    \end{align}
\end{prop}
\begin{proof}
Since $\pi_{-3}\SQM = 0$ by Assumption \ref{ass_SQM_3}, we have a canonical isomorphism
    \begin{align}
        \pi_{2}I_\Z \SQM \simeq \Hom(\pi_{-2}\SQM, \Z). 
    \end{align}
    The element $I_\Z \psi (\gamma_\K)$ is, regarded as a homomorphism via the above isomorphism, fits into the following commutative diagram. 
    \begin{align}
    \xymatrix{
    \pi_{-2}\SQM \ar@/^20pt/[rrr]^-{I_\Z \psi (\gamma_\K)} \ar[r]^-{\varphi} \ar[d]^-{\Psi} & \pi_{-2} \K \ar[rr]^-{\gamma_\K} \ar[d]^-{\Ph} && \Z \ar@{_{(}-_>}[d] \\
    0 \ar[r] & \Q \ar[rr]_-{\id} && \Q
    }
    \end{align}
    Here the commutativity of the left square comes from Assumption \ref{ass_spin_orientation}. Thus we get the desired result. 
\end{proof}

We now use the homotopy fiber sequence
\begin{align}
    I_\Z (\K/\SQM) \xrightarrow{I_\Z C\psi} I_\Z \K \xrightarrow{I_\Z \psi} I_\Z \SQM. 
\end{align}

\begin{defn}[{$\gamma_{\spin^c/\spin}$}]\label{def_beta_rel}
    By Proposition \ref{prop_vanish_anomaly_SQM}, we can choose a lift of $\gamma_{\K} \in \pi_{2}I_\Z \K$ to an element in $\pi_{2}I_\Z \K/\SQM$, i.e., an element $\gamma_{\K/\SQM}\in \pi_{2}I_\Z \K/\SQM$ so that 
    \begin{align}\label{eq_gamma_K/SQM}
        I_\Z C\psi(\gamma_{\K/\SQM})=\gamma_{\K}. 
    \end{align}
    Choose any lift, and abuse the notation to denote also by 
    \begin{align}
        \gamma_{\K/\SQM} \colon \SQM \to \Sigma^{-2}I_\Z \K/\SQM
    \end{align}
    the $\SQM$-module map given by the multiplication of the chosen element.
    We define $\gamma_{\spin^c/\spin}$ to be the following composition, 
    \begin{align}\label{eq_def_gamma_spinc/spin}
        \gamma_{\spin^c/\spin} \colon \SQM\xrightarrow{\gamma_{\K/\SQM}}\Sigma^{-2}I_\Z \K/\SQM \xrightarrow{I_\Z \rho_\rel} \Sigma^{-2}I_\Z \MSpin^c/\MSpin. 
    \end{align}
    where $\rho_\rel$ is given in \eqref{diag_K/SQM}. 
\end{defn}

\begin{rem}\label{rem_gamma_K/SQM}
   We can conveniently understand the definition of the element $\gamma_{\K/\SQM}\in \pi_{2}I_\Z \K/\SQM$ as follows. 
   We have the following commutative diagram where all the  rows and the middle column are exact, 
   \begin{align}
       \xymatrix{
       0 \ar[r]  & \Ext(\pi_{-3}\K/\SQM, \Z) \ar[r]\ar[d]^-{(C\psi)^*=0} &  \pi_{2}I_Z \K/\SQM \ar[d]^-{I_\Z C\psi}  \ar[r]^-{p} & \Hom(\pi_{-2}\K/\SQM, \Z) \ar[r] \ar[d]^-{(C\psi)^*}  &  0 \\
       0 \ar[r] & \Ext(\pi_{-3}\K, \Z)=0 \ar[r]^-{0} \ar[d]^-{\psi^* = 0}&  \pi_{2}I_Z \K \ar[r]^-{\simeq}\ar[d]^-{I_\Z\psi = 0} & \Hom(\pi_{-2}\K, \Z) \ar[r]\ar[d]^-{\psi^* = 0} &  0 \\
       0 \ar[r] & \Ext(\pi_{-3}\SQM, \Z)=0 \ar[r]^-{0} &  \pi_{2}I_Z \SQM  \ar[r]^-{\simeq} & \Hom(\pi_{-2}\SQM, \Z)\ar[r] &  0
       }. 
   \end{align}
   Using the isomorphism $\pi_{2}I_\Z \K \simeq \Hom(\pi_{-2}\K, \Z)$, the element $\gamma_{\K} \in \pi_2 I_Z\K$ was defined to be the element which corresponds to the isomorphism $\pi_{-2}\K \simeq \Z$. 
   Thus, for an element $\gamma_{\K/\SQM}\in \pi_{2}I_\Z \K/\SQM$, the following two conditions are equivalent. 
   \begin{enumerate}
       \item We have the equality \eqref{eq_gamma_K/SQM}. 
       \item Its image $p(\gamma_{\K/\SQM})$ in $\Hom(\pi_{-2}\K/\SQM, \Z)$ fits into the following commutative diagram. 
       \begin{align}
           \xymatrix{
           \pi_{-2}\K \ar[r]^-{C\psi} \ar[rd]^-{\simeq}_{\beta^{-1} \mapsto 1} & \pi_{-2}\K/\SQM  \ar[d]^-{p(\gamma_{\K/\SQM})} \\
           & \Z. 
           }
       \end{align}
   \end{enumerate}
   So we see that the requirement for an element $\gamma_{\K/\SQM}\in \pi_{2}I_\Z \K/\SQM$ in Definition \ref{def_beta_rel} is purely about the induced $\Z$-valued homomorphism. In particular we have the ambiguity of the choice, which is a torsor over $\Ext(\pi_{-3}\K/\SQM, \Z)$. 
   As we will see, since we are only interested in the $\Z$-valued homomorphism part in the following discussion, this ambiguity is irrelevant for us. 
\end{rem}

\begin{const}\label{const_pairing_general}
    For any map of spectra of the form $\alpha \colon E \to \Sigma^{s} I_\Z G$, we associate the $\Z$-valued pairing
\begin{align}\label{eq_pairing_general}
    \langle -, - \rangle_\alpha \colon \pi_{-d}E \times \pi_{d+s}G \to \Z
\end{align}
by the composition $\pi_{-d}E \xrightarrow{\alpha} \pi_{-d-s}I_\Z G \to \Hom(G_{d+s}, \Z) \hookrightarrow \Hom( \pi_{d+s}G, \Z) \simeq [G, \Sigma^{d+s}H\Q]$. 

On the other hand, the rationalization $\alpha_\Q \colon E_\Q \to \Sigma^{s}(I_\Z G)_\Q \simeq \Sigma^{s}F(G, H\Q)$ induces the $\Q$-valued pairing, 
\begin{align}\label{eq_pairing_general_Q}
    \langle -, -\rangle_\alpha^\Q \colon \pi_{-d}E_\Q \times \pi_{d+s}G_\Q \to \Q
\end{align}
by $\pi_{-d}E_\Q \xrightarrow{\alpha_\Q} [G, \Sigma^{d+s}H\Q] \simeq \Hom(G_{d+s}, \Q) $. 
\end{const}
By construction, we have the following compatibility between
the $\Z$-valued and $\Q$-valued pairings. 

\begin{lem}\label{lem_pairing_compatibility}
    In the settings of Construction \ref{const_pairing_general}, for each $e \in \pi_{-d} E$ and $f \in \pi_{d+s}F$ with rationalizations $e_\Q$ and $f_\Q$ respectively, we have
\begin{align}\label{eq_pairing_compatibility}
    \langle e, f \rangle_\alpha = \langle e_\Q, f_\Q\rangle_\alpha^\Q. 
\end{align}
\end{lem}

Thus, the values of $\Z$-valued pairing \eqref{eq_pairing_general} can be computed after rationalization, where the computation is typically easy. 
The point is that the integrality of the pairing is guaranteed. 
We also observe that the $\Z$-valued pairing only depends on the rationalization $\alpha_\Q$ of $\alpha$. 

We first apply the above construction to $\alpha=\gamma_{\K/\SQM} \colon \SQM \to \Sigma^{-2}I_\Z \K/\SQM$. 
We get a $\Z$-valued pairing
\begin{align}\label{eq_pairing_K/SQM}
    \langle -, - \rangle_{\gamma_{\K/\SQM}} \colon \pi_{-d}\SQM \times \pi_{d-2}\K/\SQM \to \Z. 
\end{align}

\begin{lem}[{The formula for $\langle -, - \rangle_{\gamma_{\K/\SQM}}$}]\label{lem_formula_pairing_K/SQM}
Let $d$ be an integer such that $d \equiv 0 \pmod 4$.  
Then in the following diagram
\begin{align}
    \xymatrix{
    \pi_{-d}\SQM \ar[d]_{\Psi}^-{\eqref{diag_K/SQM}} & \times&\pi_{d-2}\K/\SQM\ar[d]_-{\Ch/\Psi}^-{\eqref{diag_K/SQM}} \ar[rrrr]^-{\langle -, - \rangle_{\gamma_{\K/\SQM}}} &&&& \Z \ar@{_{(}-_>}[d]\\
     \pi_{-d}\Sigma^{4\Z}H\Q =\Q & \times & \pi_{d-2}\Sigma^{4\Z + 2}H\Q = \Q \quad
     \ar[rrrr]^-{
    \cdot_\Q} &&&& \Q 
     }, 
\end{align}
the two pairings are compatible, i.e., we have
    \begin{align}\label{eq_formula_pairing_K/SQM}
        \langle \mathcal{T}, \mathcal{R} \rangle_{\gamma_{\K/\SQM}} = \Psi(\mathcal{T}) \cdot_\Q (\Ch/\Psi)(\mathcal{R}) . 
    \end{align}
    for all $\mathcal{T} \in \pi_{-d}\SQM$ and $\mathcal{R} \in \pi_{d-2}\K/\SQM$. 
    Here $\cdot_\Q$ is the multiplication in $\Q$. 
\end{lem}

\begin{proof}
Let us denote by $\gamma_{\K} \colon \K \to \Sigma^{-2}I_\Z \K$ the Anderson duality of K-theory composed with the Bott periodicity. It is equivalent to the multiplication by $\gamma_{\K} \in \pi_{-2}I_\Z \K$ in Definition \ref{def_gamma_K}. 
Consider the following diagram. 
\begin{align}
\xymatrix{
\pi_{-d}\SQM \ar[d]^-{\psi}\ar@/_50pt/[dd]_(.7){\Psi} & \times&\pi_{d-2}\K/\SQM \ar[rrrr]^-{\langle -, - \rangle_{\gamma_{\K/\SQM}}} \ar@/_50pt/[dd]_(.7){\Ch/\Psi} &&&& \Z \ar@{=}[d]\\
    \pi_{-d}\K  \simeq \Z \ar@{_{(}-_>}[d]^-{\Ch} & \times&\pi_{d-2}\K \simeq \Z \ar@{_{(}-_>}[d]^-{\Ch} \ar[u]^-{C\psi} \ar[rrrr]^-{\langle -, - \rangle_{\gamma_{\K}}= (-)\cdot_\Z (-)} &&&& \Z\ar@{_{(}-_>}[d] \\
    \pi_{-d}\Sigma^{4\Z}H\Q =\Q & \times & \pi_{d-2}\Sigma^{4\Z + 2}H\Q = \Q \quad
     \ar[rrrr]^-{
    \cdot_\Q} &&&& \Q
}
\end{align}
The two triangles are commutative by the commutativity of \eqref{diag_K/SQM}. 
The equality $\langle -, - \rangle_{\gamma_{\K}}= (-)\cdot_\Z (-)$ in the middle row follows by definition of $\gamma_K$. 
The pairings in the second and third rows are compatible since the Chern character induces the canonical inclusion $\Z \hookrightarrow \Q$ in our identifications. 
The pairings in the first and second rows are also compatible, in the sense that we have
\begin{align}
   \langle \mathcal{T}, C\psi(\mathcal{V})\rangle_{\gamma_{\K/\SQM}} =
   \langle \psi(\mathcal{T}), \mathcal{V}\rangle_{\gamma_{\K}}
\end{align}
for each $\mathcal{T} \in \pi_{-d}\SQM$ and $\mathcal{V} \in \pi_{d-2}\K/\SQM$ by Remark \ref{rem_gamma_K/SQM}. 
Thus we get the compatibility of the first and third rows, which completes the proof of Lemma \ref{lem_formula_pairing_K/SQM}.  
\end{proof}

Now we apply Construction \ref{const_pairing_general} to $\alpha = \gamma_{\spin^c/\spin} \colon \SQM \to \Sigma^{-2}I_\Z \MSpin^c/\MSpin$. We get a $\Z$-valued pairing
\begin{align}\label{eq_pairing_SQM}
    \langle -, - \rangle_{\gamma_{\spin^c/\spin}} \colon \pi_{-d} \SQM \times \pi_{d-2} \MSpin^c/\MSpin \to \Z. 
\end{align}
We have the following formula for the pairing $\langle -, - \rangle_{\gamma_{\spin^c/\spin}}$. 

\begin{prop}[{the formula for the pairing $\langle -, - \rangle_{\gamma_{\spin^c/\spin}}$}]\label{prop_formula_pairing_SQM}
Let $d$ be an integer such that $d \equiv 0 \pmod 4$. 
    Then the pairing \eqref{eq_pairing_SQM}
    associated to the morphism $\gamma_{\spin^c/\spin}$ is 
    given by the following formula. 
    \begin{align}\label{eq_formula_pariring_SQM_inProp}
        \langle \mathcal{T}, [M, N]\rangle_{\gamma_{\spin^c/\spin}}
    =\Psi(\mathcal{T}) \cdot_\Q \Ind_\rel ([M, N]) . 
    \end{align}
    In particular, the pairing \eqref{eq_pairing_SQM} does not depend on the choice made in Definition \ref{def_beta_rel}. 
\end{prop}

\begin{proof}
Consider the following diagram. 
\begin{align}\label{diag_proof_formula_pairing_SQM}
    \xymatrix{
    \pi_{-d}\SQM \ar@{=}[d] & \times & \pi_{d-2}\MSpin^c/\MSpin \ar[d]_-{\rho_\rel} \ar@/_60pt/[dd]_(.7){\Ind_\rel} \ar[rrrr]^-{\langle -, - \rangle_{\gamma_{\spin^c/\spin}}} &&&& \Z \ar@{=}[d]\\
   \pi_{-d}\SQM \ar[d]_{\Psi}^-{\eqref{diag_K/SQM}} & \times&\pi_{d-2}\K/\SQM\ar[d]_-{\Ch/\Psi}^-{\eqref{diag_K/SQM}} \ar[rrrr]^-{\langle -, - \rangle_{\gamma_{\K/\SQM}}} &&&& \Z \ar@{_{(}-_>}[d]\\
     \pi_{-d}\Sigma^{4\Z}H\Q =\Q & \times & \pi_{d-2}\Sigma^{4\Z + 2}H\Q = \Q \quad
     \ar[rrrr]^-{
    \cdot_\Q} &&&& \Q 
     }
\end{align}
The pairings in the first and second rows are compatible by the definition of $\gamma_{\spin^c/\spin}$ in \eqref{eq_def_gamma_spinc/spin}. 
The compatibility of the second and third rows follow from Lemma \ref{lem_formula_pairing_K/SQM}. 
The middle triangle is commutative because of the commutativity of \eqref{diag_K/SQM}. 
Combining these, we get the desired formula \eqref{eq_formula_pariring_SQM_inProp}. This completes the proof of Proposition \ref{prop_formula_pairing_SQM}. 
\end{proof}

\begin{defn}\label{def_SQMpairing}
    By the last sentence of Proposition \ref{prop_formula_pairing_SQM}, we abbreviate the notation to denote the pairing \eqref{eq_pairing_SQM} just by 
    \begin{align}\label{eq_pairing_noindex_SQM}
    \langle -, - \rangle_{\SQM} \colon \pi_{-d} \SQM \times \pi_{d-20} \MSpin^c/\MSpin \to \Z,
\end{align}
and call it the {\it secondary $\Z$-valued pairing for $\SQM$. }. 
\end{defn}

\subsection{Proof of the lower bound $8$ (Theorem \ref{thm_SQM})}\label{subsec_proof_SQM}

In this section, we prove Theorem \ref{thm_SQM}. 
First observe that, since $\Psi$ is a morphism of ring spectra, the induced map 
\begin{align}\label{eq_Psi}
    \Psi \colon \pi_{*}\SQM \to \oplus_{k \in \Z}\Q[4k]
\end{align}
is a homomorphism of graded rings. 
The image of \eqref{eq_Psi} is a unital graded subring, and it is enough to show that the subring does not contain any invertible element in degree not divisible by $8$. 
Thus, Theorem \ref{thm_SQM} follows from the following proposition. 

\begin{prop}\label{prop_lowerbound_SQM}
Let $d$ be an integer so that $d \equiv 4 \pmod 8$. 
    If $x\in \Q$ is contained in the image of
    \begin{align}
        \Psi \colon \pi_{d}\SQM \to \pi_{d}\Sigma^{4\Z}H\Q =\Q, 
    \end{align}
    then we have $x \in 2\Z \subset \Q$. 
\end{prop}

\begin{proof}
First we observe that it is enough to show the claim for $d=-4$. 
Indeed, we know that there is a $8$-dimensional closed spin manifold $M$ such that $\Ind_{\spin}(M) = 1 \in \pi_{8}\Sigma^{4\Z}H\Q = \Q$. In particular, we see that the image of $\Psi$ contains $1$ at degree $8$, allowing us to shift the degree by $8$. So it is enough to show the claim for $d=-4$. 

    We use the pairing
    \begin{align}
        \langle-,-\rangle_{\SQM}\colon \pi_{-4}\SQM \times \pi_{2}\MSpin^c/\MSpin\to \bZ.
    \end{align}
    
    Suppose there exists $\mathcal{T} \in \pi_{-4} \SQM$ with $\Psi(\mathcal{T}) = x$. 
    We pair it with the element $[D^2, S^1_\Lie] \in \pi_2 \MSpin^c/\MSpin$ in Example \ref{ex_D2S1}. Recall from Lemma \ref{lem_Ind_rel_D2S1} that we have $\Ind_\rel([D^2, S^1_\Lie])= \frac{1}{2}$. 
    Then, by the formula \eqref{prop_formula_pairing_SQM} and the integrality of the pairing, we get
    \begin{align}
        \Z &\ni \langle \mathcal{T}, [D^2, S^1_\Lie] \rangle_{\SQM}
        = \Psi(x) \cdot_{\Q} \Ind_{\rel}([D^2, S^1_\Lie])
        = \frac{x}{2}. 
    \end{align}
    This gives the desired result. 
\end{proof}

\section{The $\ge 576$-periodicity of the spectrum $\SQFT$}\label{sec_lowerbound}

This is the main section of this paper. 
Our goal is to a geometric understanding of the lower bound $576$ of the periodicity for a class of spectra which includes $\TMF$. 

We formulate the problem as outlined in the Introduction. 
In this section, we do NOT use any knowledge of $\TMF$. 
But we allow ourselves to use the knowledge of $\MString$, $\MSpin$ and $\KO$, as well as the Ando-Hopkins-Rezk orientation \cite{AHR10} for spin manifolds $\AHR_\spin \colon \MSpin \to \KO((q))$. 
We set up the problem in Subection \ref{subsec_assumptions}. 
We start from an arbitrary $E_\infty$ ring spectrum $\SQFT$, and put assumptions which connects it with string manifolds and modular forms, as well as a vanishing assumption of homotopy group of a particular degree, namely $-21$. 
Then in the rest of this section we show that the periodicity of $\SQFT$ is at least $576$ (Theorem \ref{thm_main}). 
We construct a $\Z$-valued Anderson duality pairing (Definition \ref{def_SQFTpairing}) in Subsection \ref{subsec_pairing}. 
Then we use it to prove Theorem \ref{thm_intro} in Subsection \ref{subsec_proof_SQFT}.

\subsection{Assumptions for the spectrum $\SQFT$}\label{subsec_assumptions}
Here we list the mathematical assumptions for the spectrum $\SQFT$. 

\begin{ass}\label{ass_string_orientation}
    $\SQFT$ is an $E_\infty$ ring spectrum equipped with an $E_\infty$ ring maps
    \begin{align}
        \sigma_{\stri} &\colon \MString \to \SQFT, \\
        \varphi &\colon \SQFT \to \KO((q)), \\
        \Phi &\colon \SQFT \to H\MF^\Q, 
    \end{align}
    where $H\MF^\Q$ is the Eilenberg-Maclane spectrum with coefficient being the graded ring $\MF^Q$ with grading twice that of the modular forms. 
    These maps fit into a commutative diagram in the $(\infty, 1)$-category of $E_\infty$ ring spectra, 
    \begin{align}\label{diag_rationalization_SQFT}
        \xymatrix{
        \MString \ar[d]^-{\iota} \ar[rr]^-{\sigma_\stri} \ar@/^18pt/[rrr]^-{\Wit_{\stri}}  && \SQFT \ar[d]^-{\varphi} \ar[r]^-{\Phi} &  H\MF^\Q \ar@{_{(}-_>}[d] \\
        \MSpin \ar[rr]^-{\AHR_{\spin}} \ar@/_18pt/[rrr]_-{\Wit_{\spin}} && \KO((q)) \ar[r]^-{\mathrm{Ph}} & \Sigma^{4\Z}H\Q((q))
        }. 
    \end{align}
    where the left vertical arrow is the forgetful map, the right vertical arrow is induced by the inclusion $\MF^\Q_* \hookrightarrow \oplus_{k \in \Z}\Q((q))[2k]$ in \eqref{eq_MF_notation}, and $\AHR_\spin \colon \MSpin \to \KO((q))$ is the usual sigma-orientation for $\MSpin$ used in \cite{AHR10}. 
    Here, as written in the diagram, we define $\Wit_\stri := \Phi \circ \sigma_\stri$ and $\Wit_{\spin} := \Ph \circ \AHR_{\spin}$. 
\end{ass}

\begin{rem}
    $\TMF$ satisfies Assumption \ref{ass_string_orientation}, by setting $\sigma_\stri := \AHR_{\stri}$, the Ando-Hopkins-Rezk orientation of $\TMF$ for string manifolds \cite{AHR10}, 
    $\varphi:= \varphi_{\TMF} \colon \TMF \to \KO((q))$ to be the map induced by the restriction of the global section $\TMF:= \Gamma(\mathcal{M}_{\mathrm{ell}}, \mathcal{O}^{\mathrm{top}})$ to the smooth part of the Tate moduli $\mathcal{M}_\mathrm{ Tate} \subset \overline{\mathcal{M}_{\mathrm{ell}}}$ \cite{HillLawson},
    and $\Phi:= \Phi_\TMF \colon \TMF \to H\MF^\Q$ to be the map induced by the base change $\mathcal{M}_{\mathrm{ell}/\bC} \simeq \mathbb{H}//\mathrm{SL}(2, \Z) \to \mathcal{M}_{\mathrm{ell}}$. 
\end{rem}

\begin{rem}
    We use the notation $\Wit_{\stri}$ and $\Wit_{\spin}$ because they are the map extracting the $\Z((q))$-valued Witten genera. 
    Indeed, the morphism $\Wit_{\spin} := \Ph \circ \AHR_{\spin}$ is such a map by definition. 
    For $\Wit_\stri$, although we do not define it by the AHR orientation for string manifolds, the commutativity of \eqref{diag_rationalization_SQFT} forces us to conclude that $\Wit_\stri$ coincides with the composition $\MString \xrightarrow{\AHR_{\stri}} \TMF \xrightarrow{\Phi_{\TMF}} H\MF^\Q$, so that it is the map extracting $\Z((q))$-valued Witten genera for string manifolds.
\end{rem}

\begin{ass}\label{ass_21}
    We have 
    \begin{align}
        \pi_{-21}\SQFT = 0.  
    \end{align}
\end{ass}

\begin{rem}
    $\TMF$ satisfies Assumption \ref{ass_21} \cite[Chapter 13]{TMFBook}. 
\end{rem}

\begin{thm}\label{thm_main}
    There is no invertible elements in $\pi_n \SQFT$ for $n\not\equiv 0 \pmod{576}$. 
    In other words, the periodicity (possibly $\infty$) of the spectrum $\SQFT$ is at least $576$. 
\end{thm}

As explained in the Introduction, we translate the problem into a question on modular form images as follows. 
The upper right horizontal arrow of \eqref{diag_rationalization_SQFT_intro} induces the graded unital ring homomorphism
\begin{align}\label{eq_Phi_main}
    \Phi \colon \pi_{*}\SQFT \to \MF_{*/2} = \Z[c_4, c_6, \Delta, \Delta^{-1}]. 
\end{align}
Here we can take the codomain to be the ring of {\it integral} weakly holomorphic modular forms, since the right square of \eqref{diag_rationalization_SQFT_intro} is commutative. 
Theorem \ref{thm_main} follows from the following Proposition. 
\begin{prop}\label{prop_main}
    The periodicity (possibly $\infty$) of the graded subring
    \begin{align}
        \mathrm{Im}\left( \Phi \colon \pi_{*}\SQFT \to \MF_{*/2}\right) \subset \MF_{*/2}
    \end{align}
    is at least $576$. 
\end{prop}

The rest of this section is devoted to the proof of Proposition \ref{prop_main}. 

\subsection{Relative Witten genera}\label{subsec_relative_Wit}

As a preparation for the proof of Proposition \ref{prop_main}, we discuss the consequence of Assumption \ref{ass_string_orientation}. In particular we introduce the {\it relative Witten genera}, which appear in the formula for the Anderson duality pairings \eqref{eq_formula_pariring_inProp}. 

The commutative diagram \eqref{diag_rationalization_SQFT} induces the following commutative diagram of $\MString$-modules, where the rows are homotopy fiber sequences, 
\begin{align}\label{diag_KO((q))/SQFT}
\xymatrix{
\MString \ar[rrr]^-{\iota} \ar[d]^-{\sigma_\stri}\ar@/_50pt/[ddd]_(.2){\Wit_{\stri}}  &&& \MSpin \ar[rrr]^-{C\iota} \ar[d]^-{\AHR_\spin} \ar@/_50pt/[ddd]_(.2){\Wit_\spin}  &&& \MSpin/\MString \ar[d]^-{\sigma_\rel} \ar@/_90pt/[ddd]_(.2){\Wit_\rel} \\
\SQFT \ar[rrr]^(.3){\Phi} \ar[d]^-{\otimes \Q}\ar@/_30pt/[dd]_(.2){\Phi} &&& \KO((q)) \ar[rrr]^(.3){C\Phi} \ar[d]^-{\otimes \Q} \ar@/_30pt/[dd]_(.2){\Ph}  &&& \KO((q))/\SQFT \ar[d]^-{\otimes \Q}\ar@/_60pt/[dd]_(.3){\Ph/\Phi} \\
\SQFT_\Q \ar[rrr]^(.3){\Phi_\Q} \ar[d]^-{\Phi_\Q} &&& \KO((q))_\Q \ar[rrr]^(.3){C\Phi_\Q} \ar[d]_-{\simeq}^-{\Ph_\Q} &&& (\KO((q))/\SQFT)_\Q \ar[d]^-{(\Ph/\Phi)_\Q} \\
H\MF^\Q \ar@{^{(}-_>}[rrr] &&& \Sigma^{4\Z}H\Q((q)) \ar[rrr]_-{\mathrm{proj}} &&& \Sigma^{4\Z}H\Q((q)) / H\MF^\Q. 
}
\end{align}
We define the morphisms $\Wit_\rel$, $\sigma_\rel$ and $\Ph/\Phi$ as indicated in the diagram. 
The morphism $\Wit_\rel$ induces the map
\begin{align}\label{eq_rel_Wit}
    \Wit_\rel \colon \pi_{d}\MSpin/\MString \to \begin{cases} \frac{\Q((q))}{\MF_{d/2}^\Q} \quad \mbox{if }d \equiv 0 \pmod 4, \\
    0 \quad \mbox{otherwise}
    \end{cases}. 
\end{align}
which we call the {\it relative Witten genus}. 
Let $d \equiv 0 \pmod 4$. By the commutativity of \eqref{diag_KO((q))/SQFT}, the following diagram commutes. 
\begin{align}\label{diag_Wit_rel_and_Wit}
    \xymatrix{
    \pi_d \MSpin \ar[d]_-{C\iota \colon [M]\mapsto [M, \varnothing] } \ar[rr]^-{\Wit_{\spin}} && \Q((q)) \ar[d]^-{\mathrm{proj}}\\
\pi_d \MSpin/\MString \ar[rr]_-{\Wit_\rel} &&\frac{\Q((q))}{\MF_{d/2}^\Q} . 
    }
\end{align}

\begin{ex}[{The element $[D^4, S^3_\Lie] \in \pi_{4}\MSpin/\MString$ and its relative Witten genus}]\label{ex_D4S3}
Let us consider $d=4$ case. The homotopy fiber exact sequence for bordism groups looks like
\begin{align}\label{diag_MSpin/MString_deg4}
    \xymatrix{
    \pi_4 \MString \ar[r]  & \pi_4 \MSpin \ar[r] &  \pi_4\MSpin/\MString \ar[r]^-{\del} & \pi_3 \MString \ar[r] & \pi_3 \MSpin \\
    0 \ar[r] \ar[u]^-{\simeq}& \Z \ar[r]^-{\times 24} \ar[u]^-{\simeq}_{\cdot[K3]}& \Z \ar[r]^-{\mod 24}  \ar[u]^-{\simeq}_{\cdot [D^4, S^3_\Lie]} & \Z/24 \ar[r] \ar[u]^-{\simeq}_{\cdot [S^3_\Lie]} & 0\ar[u]^-{\simeq}  .
    }
\end{align}
this fact is well-known, for example see \cite[Example D.6]{Tachikawayamashita2021}. 
As explained there, element $[D^4, S^3_\Lie] \in \pi_4 \MSpin/\MString$ consists of four dimensional disk $D^4$ equipped with the trivial spin structure, and on the boundary $\del D^4 \simeq S^3$ the spin structure is lifted to the string structure coming from the Lie group framing of $S^3 \simeq SU(2)$. 
We can construct a framing outside $24$ points on the $K3$ surface so that we get the relations
\begin{align}\label{eq_D4S3_generates}
    \del([D^4, S^3_\Lie]) = [S^3_\Lie], \quad \mbox{and} \quad [K3] = 24[D^4, S^3_\Lie], 
\end{align}
which verifies the diagram \eqref{diag_MSpin/MString_deg4}. 

The relative Witten genus of the element $[D^4, S^3_\Lie]\in \pi_4\MSpin/\MString$ is computed as follows. 
 \begin{lem}\label{lem_Ind_rel_D4S3}
        We have the following equality in $\frac{\Q((q))}{\MF_{2}}$, 
        \begin{align}
            \Wit_{rel}([D^4, S^3_\Lie]) = \frac{E_2(q)}{12} \mod \MF^\Q_2, 
        \end{align}
        where $E_2(q)$ is the second Eisenstein series with the normalization so that $E_2(q) = 1 + O(q)$. 
    \end{lem}
    \begin{proof}
        We use the commutativity of \eqref{diag_Wit_rel_and_Wit} and the knowledge of the usual Witten genus for the $K3$-surface, 
        \begin{align}
            \Wit_\spin([K3])= 2E_2(q). 
        \end{align}
        We have
        \begin{align}
            \Wit_\rel([D^4, S^3_\Lie]) &= \frac{1}{24} \cdot \Ind_{\rel}(24[D^4, S^3_\Lie])\\
            &= \frac{1}{24} \cdot \Wit_{\spin}([K3]) \quad (\rm{mod} \  \MF^\Q_2) \\
            &= \frac{2E_2(q)}{24} \quad (\rm{mod} \ \MF^\Q_2).  
        \end{align}
        Here, for the second equality we used \eqref{eq_D4S3_generates}. 
    \end{proof}
\end{ex}

We use the following product formula for the relative Witten genus. 
 Recall that $\MSpin/\MString$ is canonically an $\MString$-module. 
 \begin{lem}\label{lem_product_Wit}
     For any $[L] \in \pi_{d_1}\MString$ and $[M, N] \in \pi_{d_2}\MSpin/\MString$, we have
     \begin{align}
         \Wit_\rel([L] \cdot [M, N]) = \Wit([L]) \cdot \Wit_\rel([M, N]). 
     \end{align}
     In particular, if $d_1$ or $d_2$ is not divisible by $4$, we have $\Wit([L] \cdot [M, N]) = 0$. 
 \end{lem}
 \begin{proof}
     Directly follows from the fact that \eqref{diag_KO((q))/SQFT} is a commutative diagram of $\MString$-modules.  
 \end{proof}

\subsection{Review of the secondary pairing in \cite{tachikawa2023anderson}}\label{subsec_pairing}
In \cite{tachikawa2023anderson}, they considered a morphism $\TMF \to \Sigma^{-20}\MSpin/\MString$ which arises as the secondary morphism of the anomaly cancellation in heterotic string theory, and investigated into the pairing between homotopy groups of $\TMF$ and $\MSpin/\MString$. 
Our key observation here is that their construction only uses the fact that the spectrum $\TMF$ satisfies the assumptions listed in Subsection \ref{subsec_assumptions}. 
In this subsection, we reproduce the construction of the secondary $\Z$-valued pairing in \cite{tachikawa2023anderson}, with ``$\TMF$'' replaced to ``$\SQFT$''. 

Let $\SQFT$ be a spectrum satisfying the assumptions in Subsection \ref{subsec_assumptions}. 
In this subsection, we construct a morphism (Definition \ref{def_alpha_spin/string})
\begin{align}
    \alpha_{\spin/\stri} \colon \SQFT \to \Sigma^{-20}I_\Z \MSpin/\MString,  
\end{align}
which induces a pairing (Construction \ref{const_pairing_general})
\begin{equation}
\langle-,-\rangle_\SQFT \colon \pi_{-4k}\SQFT \times \pi_{4k-20}\MSpin/\MString\to \bZ,
\label{geompair}
\end{equation}
Then we show in Proposition \ref{prop_formula_pairing} that we have the following formula for the pairing, 
\begin{align}\label{eq_formula_pariring}
    \langle \mathcal{T}, [M, N]\rangle_\SQFT
    =\left. \frac{1}{2}\Delta\cdot \Phi(\mathcal{T}) \cdot \Wit_\rel([M, N])\right|_{q^0}. 
\end{align}
Here $\Wit([M, N]) \in \frac{\Q((q))}{\MF_{2k-10}}$ is the relative Witten genus introduced in \eqref{eq_rel_Wit}. 
So the important point is the well-definedness, including the integrality, of the formula \eqref{eq_formula_pariring}. 

\begin{defn}[{\cite[Definition 3.8]{tachikawa2023anderson}}]\label{def_alpha_KO}
    We define an element $\alpha_{\KO((q))} \in \pi_{20}I_\Z \KO((q)) \simeq \Hom(\pi_{-20}\KO((q)), \Z)$ by the formula
\begin{align}
    \alpha_{\KO((q))} \colon \pi_{-20}\KO((q)) \simeq 2\Z((q)) \to \Z, \\
    \phi \mapsto \frac{1}{2}\Delta \phi|_{q^0}. \notag
\end{align}
\end{defn}

The following simple fact about modular forms is a key ingredient in the construction of the pairing. 
\begin{fact}[{\cite[Lemma 3.14]{Tachikawayamashita2021}}]\label{fact_MF-2}
    The constant coefficient (i.e., the coefficient of $q^0$) of any element $\phi(q) \in \MF_{-2} \subset \Z((q))$ is zero. 
\end{fact}

The following is the main theorem of \cite{Tachikawayamashita2021}, with their ``$\TMF$'' replaced by ``$\SQFT$''. 
\begin{prop}[{$\SQFT$-version of \cite[Theorem 3.17]{Tachikawayamashita2021}, \cite[Theorem 3.9]{tachikawa2023anderson}}]\label{prop_vanish_anomaly}
Let $I_\Z \varphi \colon I_\Z \KO((q)) \to I_\Z \SQFT$ denote the Anderson dual to the middle vertical arrow in \eqref{diag_rationalization_SQFT}. 
    We have the following equality in $\pi_{20}I_\Z \SQFT$. 
    \begin{align}
        I_\Z \varphi (\alpha_{\KO((q))}) = 0. 
    \end{align}
\end{prop}
\begin{proof}
    Since $\pi_{-21}\SQFT = 0$ by Assumption \ref{ass_21}, we have a canonical isomorphism
    \begin{align}
        \pi_{20}I_\Z \SQFT \simeq \Hom(\pi_{-20}\SQFT, \Z). 
    \end{align}
    The element $I_\Z \varphi (\alpha_{\KO((q))})$ is regarded as a homomorphism via the above isomorphism, fits into the following commutative diagram. 
    \begin{align}
    \xymatrix{
    \pi_{-20}\SQFT \ar@/^30pt/[rrr]^-{I_\Z \varphi(\alpha_{\KO((q))})} \ar[r]^-{\varphi} \ar[d]^-{\Phi} & \pi_{-20} \KO((q)) \ar[rr]^-{\alpha_{\KO((q))}} \ar[d]^-{\Ph} && \Z \ar@{_{(}-_>}[d] \\
    \MF_{-10}^\Q \ar@{^{(}-_>}[r] & \Q((q)) \ar[rr]_-{\frac{1}{2}\Delta \cdot -|_{q^0}} && \Q
    }
    \end{align}
    Here the commutativity of the left square comes from Assumption \ref{ass_string_orientation}. 
    By Fact \ref{fact_MF-2}, the composition of bottom rows are zero. By the commutativity of the diagram, we get the desired result.  
\end{proof}

We now use the homotopy fiber sequence
\begin{align}
    I_\Z \KO((q))/\SQFT \xrightarrow{I_\Z C\varphi} I_\Z \KO((q)) \xrightarrow{I_\Z \varphi} I_\Z \SQFT. 
\end{align}
\begin{defn}[{$\alpha_{\spin/\stri}$}]\label{def_alpha_spin/string}
    By Proposition \ref{prop_vanish_anomaly}, we can choose a lift of $\alpha_{\KO((q))} \in \pi_{20}I_\Z \KO((q))$ to an element in $\pi_{20}I_\Z \KO((q))/\SQFT$, i.e., an element $\alpha_{\KO((q))/\SQFT} \in \pi_{20}I_\Z \KO((q))/\SQFT$ so that
    \begin{align}\label{eq_alpha_KO((q))/SQFT}
         I_\Z C\varphi(\alpha_{\KO((q))/\SQFT})=\alpha_{\KO((q))}. 
    \end{align}
    Choose any lift, and denote by 
    \begin{align}
        \alpha_{\KO((q))/\SQFT} \colon \SQFT \to \Sigma^{-20}I_\Z \KO((q))/\SQFT
    \end{align}
    the $\SQFT$-module map given by the multiplication of the chosen element.
    We define $\alpha_{\spin/\stri}$ to be the following composition, 
    \begin{align}\label{eq_def_alpha_spin/stri}
        \alpha_{\spin/\stri} \colon \SQFT \xrightarrow{\alpha_{\KO((q))/\SQFT} }\Sigma^{-20}I_\Z \KO((q))/\SQFT \xrightarrow{I_\Z \sigma_\rel} \Sigma^{-20}I_\Z \MSpin/\MString. 
    \end{align}
\end{defn}

\begin{rem}\label{rem_alpha_rel}
    The analogous remark to Remark \ref{rem_gamma_K/SQM} applies here, and we can conveniently understand the definition of the element $\alpha_{\KO((q))/\SQFT} \in \pi_{20}I_\Z \KO((q))/\SQFT$ as follows. 
    We have the exact sequence
    \begin{align}
        0 \to \Ext(\pi_{-21}\KO((q))/\SQFT) \to \pi_{20}I_\Z \KO((q))/\SQFT \xrightarrow[]{p} \Hom(\pi_{-20}\KO((q))/\SQFT, \Z) \to 0. 
    \end{align}
    By exactly the same argument as Remark \ref{rem_gamma_K/SQM}, for an element $\alpha_{\KO((q))/\SQFT} \in \pi_{20}I_\Z \KO((q))/\SQFT$, the following two conditions are equivalent. 
    \begin{enumerate}
        \item We have \eqref{eq_alpha_KO((q))/SQFT}. 
        \item Its image $p(\alpha_{\KO((q))/\SQFT})$ in $\Hom(\pi_{-20}\KO((q))/\SQFT, \Z)$ fits into the following commutative diagram. 
       \begin{align}
           \xymatrix{
           \pi_{-20}\KO((q)) \ar[r]^-{C\psi} \ar[rd]_-{\left. \frac{1}{2} \Delta \cdot (-)\right|_{q^0}} & \pi_{-20}\KO((q))/\SQFT  \ar[d]^-{p(\alpha_{\KO((q))/\SQFT})} \\
           & \Z. 
           }
       \end{align}
    \end{enumerate}
    In particular we have the ambiguity of the choice of the lift, which is a torsor over $\Ext(\pi_{-21}\KO((q))/\SQFT, \Z)$. 
   As we will see, since we are only interested in the $\Z$-valued homomorphism part in the following discussion, this ambiguity is irrelevant for us. 
\end{rem}

Applying Construction \ref{const_pairing_general} to $\alpha=\alpha_{\KO((q))/\SQFT} \colon \SQFT \to \Sigma^{-20}I_\Z \KO((q))/\SQFT$. 
We get a $\Z$-valued pairing
\begin{align}\label{eq_pairing_KO((q))/SQFT}
    \langle -, - \rangle_{\alpha_{\KO((q))/\SQFT}} \colon \pi_{-d}\SQFT \times \pi_{d-20}\KO((q))/\SQFT \to \Z. 
\end{align}

\begin{lem}[{The formula for $\langle -, - \rangle_{\alpha_{\KO((q))/\SQFT}}$}]\label{lem_formula_pairing_KO((q))/SQFT}
Let $d$ be an integer such that $d \equiv 0 \pmod 4$. 
In the following diagram
\begin{align}\label{diag_pairing_KO/SQFT}
    \xymatrix{
    \pi_{-d}\SQFT \ar[d]_{\Phi}^-{\eqref{diag_KO((q))/SQFT}} & \times&\pi_{d-20}\KO((q))/\SQFT\ar[d]_-{\Ph/\Phi}^-{\eqref{diag_KO((q))/SQFT}} \ar[rrr]^-{\langle -, - \rangle_{\alpha_{\KO((q))/\SQFT}}} &&& \Z \ar@{_{(}-_>}[d]\\
     \pi_{-d}H\MF^\Q =\MF^\Q_{-d/2} & \times & \pi_{d-20}\Sigma^{4\Z}H\Q((q))/H\MF^\Q \simeq \frac{\Q((q))}{\MF^\Q_{(d-20)/2}} \quad
     \ar[rrr]^-{\left. \frac{1}{2}\Delta \cdot (-) \cdot (-) \right|_{q^0}}_-{\rm{Fact }\ \ref{fact_MF-2}}  &&& \Q 
     }, 
\end{align}
the two pairings are compatible, i.e., we have
    \begin{align}\label{eq_formula_pairing_KO/SQFT}
        \langle \mathcal{T}, \mathcal{R} \rangle_{\alpha_{\KO((q))/\SQFT}} =\left. \frac{1}{2}\Delta \cdot \Phi(\mathcal{T}) \cdot (\Ph/\Phi)(\mathcal{R}) \right|_{q^0}. 
    \end{align}
    for all $\mathcal{T} \in \pi_{-d}\SQFT$ and $\mathcal{R} \in \pi_{d-20}\KO((q))/\SQFT$. 
    Here, the well-definedness of the pairing in the bottom row of \eqref{diag_pairing_KO/SQFT} is guaranteed by Fact \ref{fact_MF-2}. 
\end{lem}

\begin{proof}
Let us denote by $\alpha_{\KO((q))} \colon \KO((q)) \to \Sigma^{-20}I_\Z \KO((q))$ be the $\KO((q))$-module morphism given by the multiplication by $\alpha_{\KO((q))} \in \pi_{-20}I_\Z \KO((q))$ in Definition \ref{def_alpha_KO}. 
Consider the following diagram. 
\begin{align}
\xymatrix{
\pi_{-d}\SQFT \ar[d]_{\Phi}^-{\eqref{diag_KO((q))/SQFT}}\ar@/_80pt/[ddd]_(.7){\Phi}  & \times&\pi_{d-20}\KO((q))/\SQFT \ar@/_90pt/[ddd]_(.7){\Ph/\Phi} \ar[rrr]^-{\langle -, - \rangle_{\alpha_{\KO((q))/\SQFT}}} &&& \Z \ar@{=}[d]\\
    \pi_{-d}\KO((q)) \ar@{_{(}-_>}[d]^-{\Ph} & \times&\pi_{d-20}\KO((q))  \ar@{_{(}-_>}[d]^-{\Ph} \ar[u]^-{C\psi} \ar[rrr]^-{\langle -, - \rangle_{\alpha_{\KO((q))}}} &&& \Z\ar@{_{(}-_>}[d] \\
    \pi_{-d}\Sigma^{4\Z}H\Q((q)) = \Q((q))  & \times&\pi_{d-20}\Sigma^{4\Z}H\Q((q)) = \Q((q))  \ar[d]^-{\mathrm{proj}}  \ar[rrr]^-{ \left. \frac{1}{2}\Delta \cdot (-) \cdot (-)\right.|_{q^0}} &&& \Q\ar@{=}[d] \\
    \pi_{-d}H\MF^\Q =\MF^\Q_{-d/2}\ar@{_{(}-_>}[u]^-{} & \times & \pi_{d-20}\Sigma^{4\Z}H\Q((q))/H\MF^\Q \simeq \frac{\Q((q))}{\MF^\Q_{(d-20)/2}} \quad
     \ar[rrr]^-{\left. \frac{1}{2}\Delta \cdot (-) \cdot (-) \right|_{q^0}}_-{\rm{Fact }\ \ref{fact_MF-2}} &&& \Q 
    }
\end{align}
The two squares (i.e., the left and middle columns) are commutative by the commutativity of \eqref{diag_KO((q))/SQFT}. 
The pairings in the first and second rows are compatible by Remark \ref{rem_alpha_rel}. 
The pairings in the second and third rows are compatible by Definition \ref{def_alpha_KO}. 
The compatibility between the third and fourth rows is automatic. 
Thus we get the compatibility of the first and fourth rows, which completes the proof of Lemma \ref{lem_formula_pairing_KO((q))/SQFT}.  
\end{proof}

Next we apply Construction \ref{const_pairing_general} to $\alpha = \alpha_{\spin/\stri} \colon \SQFT \to \Sigma^{-20}I_\Z \MSpin/\MString$. We get a $\Z$-valued pairing
\begin{align}\label{eq_pairing}
    \langle -, - \rangle_{\alpha_{\spin/\stri}} \colon \pi_{-d} \SQFT \times \pi_{d-20} \MSpin/\MString \to \Z. 
\end{align}

\begin{prop}[{the formula for $\langle -, - \rangle_{\alpha_{\spin/\stri}}$}]\label{prop_formula_pairing}
Let $d$ be an integer such that $d \equiv 0 \pmod 4$. 
The pairing \eqref{eq_pairing}
    associated to the morphism $\alpha_{\spin/\stri}$ is given by the following formula. 
    \begin{align}\label{eq_formula_pariring_inProp}
        \langle [\mathcal{T}], [M, N] \rangle_{\alpha_{\spin/\stri}} 
        =\left.\frac{1}{2}\Delta \cdot \Phi([\mathcal{T}]) \cdot \Wit_\rel([M, N]) \right|_{q^0}
    \end{align}
    Here, $\Phi([\mathcal{T}]) \in \pi_{-d}H\MF^\Q = \MF^\Q_{-d/2}$ and $\Wit([M, N]) \in \frac{\Q((q))}{\MF_{(d-20)/2}^\Q}$ is the relative Witten genus in \eqref{eq_rel_Wit}. 
    The right hand side of \eqref{eq_formula_pariring_inProp} is well-defined because of Fact \ref{fact_MF-2}. 
    In particular, the pairing \eqref{eq_pairing} does not depend on the choice made in Definition \ref{def_alpha_spin/string}. 
\end{prop}

\begin{proof}
Consider the following diagram. 
\begin{align}\label{diag_proof_formula_pairing_SQFT}
    \xymatrix{
    \pi_{-d}\SQFT \ar@{=}[d] & \times & \pi_{d-20}\MSpin/\MString \ar[d]_-{\sigma_\rel} \ar@/_60pt/[dd]_(.7){\Wit_\rel} \ar[rrr]^-{\langle -, - \rangle_{\alpha_{\spin/\stri}}} &&& \Z \ar@{=}[d]\\
   \pi_{-d}\SQFT \ar[d]_{\Phi}^-{\eqref{diag_KO((q))/SQFT}} & \times&\pi_{d-20}\KO((q))/\SQFT\ar[d]_-{\Ph/\Phi}^-{\eqref{diag_KO((q))/SQFT}} \ar[rrr]^-{\langle -, - \rangle_{\alpha_{\KO((q))/\SQFT}}} &&& \Z \ar@{_{(}-_>}[d]\\
     \pi_{-d}H\MF^\Q =\MF^\Q_{-d/2} & \times & \pi_{d-20}\Sigma^{4\Z}H\Q((q))/H\MF^\Q \simeq \frac{\Q((q))}{\MF^\Q_{(d-20)/2}} \quad
     \ar[rrr]^-{\left. \frac{1}{2}\Delta \cdot (-) \cdot (-) \right|_{q^0}}_-{\rm{Fact }\ \ref{fact_MF-2}}  &&& \Q
     }
\end{align}
The pairings in the first and second rows are compatible by the definition of $\alpha_{\spin/\stri}$ in \eqref{eq_def_alpha_spin/stri}. 
The compatibility of the second and third rows follow from Lemma \ref{lem_formula_pairing_KO((q))/SQFT}. 
The middle triangle is commutative because of the commutativity of \eqref{diag_KO((q))/SQFT}. 
Combining these, we get the desired formula \eqref{eq_formula_pariring_inProp}. This completes the proof of Proposition \ref{prop_formula_pairing}. 
\end{proof}

\begin{defn}\label{def_SQFTpairing}
    By the last sentence of Proposition \ref{prop_formula_pairing}, we abbreviate the notation to denote the pairing \eqref{eq_pairing} just by 
    \begin{align}\label{eq_pairing_noindex}
    \langle -, - \rangle_\SQFT \colon \pi_{-d} \SQFT \times \pi_{d-20} \MSpin/\MString \to \Z,
\end{align}
and call it the {\it secondary $\Z$-valued pairing}. 
\end{defn}

\subsection{Proof of the lower bound 576 (Theorem \ref{thm_main})}\label{subsec_proof_SQFT}

In this section, we prove Proposition \ref{prop_main}, which implies the main Theorem \ref{thm_main}. 
As an illustration of our strategy, we showed that the periodicity should be strictly more than $24$ (Proposition \ref{prop_1_intro}) in the Introduction. 
Although it is logically not necessary in the proof of Proposition \ref{prop_main} in this section, we recommend the reader to follow the proof there first. 

Recall that we are allowing ourselves to use the knowledge of string manifolds freely. Also the Witten genus $\Wit_\stri([M]) \in \mf_*$ for a string manifold $M$ is a differential geometric invariant which we are free to use. 
In particular we use the following fact on Witten genera. 

\begin{fact}\label{fact_Wit_image}
\label{fact:imageZ}
The image of the ring homomorphism $\Wit_{\stri}:\pi_{4\bullet}\MString\to \MF_{2\bullet}$ 
has a $\bZ$-basis given by \begin{equation}
a_{i,j,d} c_4^i c_6^j \Delta^d, \qquad \text{$i\ge 0$; $j=0,1$; $d\in\bZ$}
\end{equation}where \begin{equation}
a_{i,j,d} = \begin{cases}
24/\gcd(24,d) & \text{if $i=j=0$},\\
2 & \text{if $j=1$},\\
1 & \text{otherwise}.
\end{cases}\label{hop}
\end{equation}
\end{fact}

\begin{rem}
The factor $2$ when $j=1$ in \eqref{hop} is due to the multiplication by 2 in $\pi_{8k+4}\KO\simeq \bZ \to \pi_{8k+4} \K\simeq \bZ$ and is not very surprising. 
The nontrivial information is therefore mainly in the $i=j=0$ case.
\end{rem}

\begin{proof}
    This follows from the surjectivity of $\pi_* \MString \to \pi_* \tmf$ (\cite{devalapurkar2020andohopkinsrezk}) and the corresponding fact about the image of $\pi_* \tmf \to \mf_{*/2}$ (\cite{Hopkins2002}). 
\end{proof}


By Fact \ref{fact_Wit_image}, we know that $\Delta^{24}$ is contained in the image $\Phi \colon \pi_{-24 \times 24}\SQFT \to \MF_{-24 \times 12}$. 
This allows us to reduce the proof to the following two propositions. 

\begin{prop}\label{prop_16}
    If $k \Delta^{-16}$ is contained in the image of $\Phi \colon \pi_{-16 \times 24}\SQFT \to \MF_{-16 \times 12}$, 
    then we have $3|k$. 
\end{prop}

\begin{prop}\label{prop_lowerbound_12}
    If $k \Delta^{-12}$ is contained in the image of $\Phi \colon \pi_{-12 \times 24}\SQFT \to \MF_{-12 \times 12}$, 
    then we have $2|k$. 
\end{prop}

\begin{proof}[Proof of Proposition \ref{prop_main} assuming Propositions \ref{prop_16} and \ref{prop_lowerbound_12}]
In this proof we use the notation for the graded unital subring
\begin{align}
    R_* := \mathrm{Im}\left( \Phi \colon \pi_{*}\SQFT \to \MF_{*/2}\right) \subset \MF_{*/2}. 
\end{align}
Since invertible elements in the graded ring $\MF_{*} = \Z[c_4, c_6, \Delta, \Delta^{-1}]$ are of exactly those the form $\pm\Delta^{d}$ for $d \in \Z$, any invertible element of $R_*$ is also of that form. In particular, the periodicity is multiple of $24$. 

Suppose that $R_*$ has periodicity $24d$ for $0< d < 24$. By above, we have $\Delta^{-d} \in R_{24d}$. 
The case $d = 8$ or $d=16$ directly contradicts Proposition \ref{prop_16}. 
If $d \neq 8, 16$, we have some positive integer $n$ and $m$ such that $nd= -12 + 24m$. 
Also note that $\Delta^{24} \in R_{24 \times 24}$ by Fact \ref{fact:imageZ} and \eqref{diag_rationalization_SQFT}. 
Then we have
\begin{align}
    R_{-12 \times 24} \ni (\Delta^{d})^n \cdot (\Delta^{24})^m = \Delta^{-12}, 
\end{align}
contradicting Proposition \ref{prop_lowerbound_12}. 
This completes the proof of Proposition \ref{prop_main} assuming Propositions \ref{prop_16} and \ref{prop_lowerbound_12}. 
\end{proof}

Thus we are left to prove Propositions \ref{prop_16} and \ref{prop_lowerbound_12}. 
As we illustrated in the proof of Proposition \ref{prop_1_intro} in the Introduction, our strategy is to use the integrality of the pairing
\begin{align}
    \langle-,-\rangle_\SQFT: \pi_{-d \times 24}\SQFT \times \pi_{d \times 24 -20}\MSpin/\MString\to \bZ.
\end{align}
for $d=16, 12$. 
Then the question boils down to the construction of interesting elements in $\pi_{d\times 24 - 20}\MSpin/\MString$ which guarantee the desired divisibility. 
The construction is easy for the case for $d=16$, as follows.

\begin{proof}[Proof of Proposition \ref{prop_16}]
    We use the pairing
    \begin{align}
        \langle-,-\rangle_\SQFT: \pi_{-16 \times 24}\SQFT \times \pi_{15 \times 24 + 4}\MSpin/\MString\to \bZ.
    \end{align}
    By Fact \ref{fact_Wit_image}, we have a $15 \times 24$-dimensional string manifold $M$ with $\Wit([M]) = 8\Delta^{15}$. 
    We also use the class $[D^4, S^3_\Lie] \in \pi_4\MSpin/\MString$ in Example \ref{ex_D4S3}. 
    Recall that we have computed the relative Witten genus of $[D^4, S^3_\Lie]$ in Lemma \ref{lem_Ind_rel_D4S3}. 
    Using the $\MString$-module structure of $\MSpin/\MString$, we get a class $[D^4, S^3_\Lie] \cdot [M] \in \pi_{15 \times 24 + 4}\MSpin/\MString$. 
    \begin{align}
        \Wit([D^4, S^3_\Lie] \cdot [M])
        &= \Wit([D^4, S^3_\Lie]) \Wit([M]) \\
        &= \left(\frac{1}{12} + O(q) \right) \cdot 8\Delta^{15}. 
    \end{align}
    
    Suppose there exists $\mathcal{T} \in \pi_{-16 \times 24} \SQFT$ with $\Phi(\mathcal{T}) = k \Delta^{-16}$. 
    Then, by the formula \eqref{eq_formula_pariring} and the integrality of the pairing, we get
    \begin{align}
        \Z &\ni \left.\frac{1}{2}\Delta \cdot k \Delta^{-16} \cdot\left(\frac{1}{12} + O(q) \right) \cdot 8\Delta^{15} \right|_{q^0} 
        = \frac{k}{3}. 
    \end{align}
    This gives the desired result. 
\end{proof}

\begin{rem}\label{rem_12}
    The reader may wonder if we could apply the same construction of manifolds to prove Proposition \ref{prop_lowerbound_12}, using an element in $\pi_{11 \times 24 + 4}\MSpin/\MString$ of the form $[D^4, S^3_\Lie] \cdot [M]$ with $[M] \in \pi_{11 \times 24}\MString$. Unfortunately, since $\gcd(11, 24)=1$, Fact~\ref{fact_Wit_image} implies that the minimum multiple of $\Delta^{11}$ which appear as the Witten genus of string manifold is $24$. If we used that manifold, we only get
    \begin{align}
     \bZ \ni \left.\frac{1}{2}\Delta \cdot k \Delta^{-12} \cdot\left(\frac{1}{12} + O(q) \right) \cdot 24\Delta^{11} \right|_{q^0} 
        = k, 
    \end{align}
    which gives no interesting consequence. We can also easily prove that any other choice of element $[M] \in \pi_{11 \times 24}\MString$ does not work as well. 
\end{rem}

Now we are left to prove Proposition \ref{prop_lowerbound_12}. 
As explained in Remark \ref{rem_12}, the most naive construction does not work. 
Rather, we will construct manifolds which realize {\it Toda brackets} in Thom spectra\footnote{Such constructions are suggested to the authors by S. Devalapurkar. }. The general construction explained in Appendix \ref{app:Toda}. 

We use the common notations for stable homotopy elements. 
In particular we are concerned with generators in degree $1$, $3$ and $20$. Let us summarize the notations. 
\begin{itemize}
    \item We have $\pi_1 \mathbb{S} \simeq \Z/2$, and $[S^1_\Lie]$ is the generator. We have $[S^1_\Lie] = \eta \in \pi_1 \mathbb{S}_{(2)}$. 
    \item We have $\pi_3 \mathbb{S} \simeq \Z/24$, and $[S^3_\Lie]$ is a generator. We have $[S^3_\Lie] = \nu \in \pi_3 \mathbb{S}_{(2)}$ and $[S^3_\Lie] = \alpha_1 \in \pi_3 \mathbb{S}_{(3)}$. 
    \item We have $\pi_{20}\mathbb{S} \simeq \Z/24$. Let us fix a $20$-dimensional stably framed manifold $\bar{K}$ representing a generator, which we take to satisfy $[\bar{K}] = \bar{\kappa} \in \pi_{20}\mathbb{S}_{(2)}$ and $[\bar{K}] = \beta_1^2 \in \pi_{20}\mathbb{S}_{(3)}$. 
\end{itemize}

We use the following facts about those elements and their Hurewicz images. 

\begin{fact}[{\cite{devalapurkar2020andohopkinsrezk}}]\label{fact_kappabar}
    \begin{enumerate}
        \item We have
        \begin{align}\label{eq_fact_nu_kappabar}
            [S^3_\Lie] \cdot [\bar{K}] = 0 \in \pi_{23}\MString. 
        \end{align} 
        Note that we have $[S^3_\Lie] \cdot [\bar{K}] \neq 0 \in \pi_{23}\bS$. 

        \item By (1) and $24[S^3_\Lie] = 0 \in \pi_3\bS$, we can form the Toda bracket $\left\langle 24, [S^3_\Lie], [\bar{K}] \right\rangle_{\tmf/\mathbb{S}}$ (see Appendix \ref{app:Toda}). We have
       \begin{align}\label{eq_Toda_Delta_tmf}
    [24\Delta] = \left\langle 24, [S^3_\Lie], [\bar{K}] \right\rangle_{\tmf/\mathbb{S}} \in \frac{\pi_{24}\tmf}{\pi_4 \tmf \cdot \bar{\kappa} + 24 \pi_{24}\tmf} = \frac{\pi_{24}\tmf}{ 24 \cdot \pi_{24}\tmf}.   
\end{align}
        Here we are using the injectivity of $\pi_{24}\tmf \to \mf_{12}$ to identify
        \begin{align}
            \pi_{24}\tmf \simeq \mathrm{im}(\pi_{24}\tmf \to \mf_{12}) = 24\Delta \cdot \Z + c_4^3 \cdot \Z \subset \mf_{12}. 
        \end{align}
        The class $[24\Delta]$ denotes $24\Delta$ modulo $24 \cdot(24\Delta \cdot  \Z + c_4^3 \cdot \Z)$. 
    \end{enumerate}
\end{fact}

\begin{proof}
    (1) follows from $\nu \bar{\kappa} = 0 \in \pi_{23}\MString_{(2)}$ \cite[Theorem 4.9]{devalapurkar2020andohopkinsrezk} and $\alpha_1 \beta_1^2 = 0 \in \pi_{23}\MString_{(3)}$ \cite[Theorem 5.6]{devalapurkar2020andohopkinsrezk}. (In \cite{devalapurkar2020andohopkinsrezk}, the corresponding equality is proved on a spectrum $B$ which factors the $\MString$-Hurewicz map as $\bS \to B \to \MString$.) 
    The corresponding products does not vanish on $\bS$ \cite{Ranavel86}. 

    (2) follows from \cite[Corollary 4.6 and Proposition 5.7]{devalapurkar2020andohopkinsrezk}, which shows the corresponding Toda bracket relation in the Adams-Novikov spectral sequences for $\tmf$. By the Toda bracket convergence theorem in \cite{belmont2021toda} applied to the Adams-Novikov spectral sequence, we get the desired result. 
\end{proof}

First, we construct the following Toda bracket manifolds (see Appendix~\ref{app:Toda} and in particular Proposition~\ref{prop_Toda_htpy=mfd} therein). 

\begin{const}[{$L_{24}^\stri$, $L_{24}^\stri$ and $L_{25}^\stri$}]
Let us fix
\begin{itemize}
    \item a string nullbordism $U^\stri_{S^3 \bar{K}}$ of $S^3_\Lie \times \bar{K}$ (which is possible by \eqref{eq_fact_nu_kappabar}), 
    \item a stably framed nullbordism $U_{24S^3}^\fr$ of $24S^3_\Lie$, 
    \item a stably framed nullbordism $U_{S^1S^3}^\fr$ of $S^1_\Lie \times S^3_\Lie$ and
    \item a spin nullbordism $U_{S^3}^\spin$ of $S^3_\Lie$. 
\end{itemize}
Using them, we construct the following manifolds. 
\begin{enumerate}
    \item A $24$-dimensional string manifold $L_{24}^\stri$ which is given by
    \begin{align}\label{eq_def_L24stri}
        L_{24}^\stri := U_{24S^3}^\fr \times \bar{K} \cup_{24S^3_\Lie \times \bar{K}} \overline{24 U^\stri_{S^3 \bar{K}}}. 
    \end{align}
    This manifold satisfies 
    \begin{align}
        [L_{24}^\stri] = \left\langle 24, [S^3_\Lie], [\bar{K}] \right\rangle_{\MString/\mathbb{S}} \in \frac{\pi_{24}\MString}{\pi_4 \mathbb{S} \cdot [\bar{K}] + 24 \pi_{24}\MString} = \frac{\pi_{24}\MString}{ 24 \cdot \pi_{24}\MString}. 
    \end{align}
    
    \item A $24$-dimensional spin manifold $L_{24}^\spin$ which is given by
    \begin{align}\label{eq_def_L24spin}
        L_{24}^\spin := U_{S^3}^\spin \times \bar{K} \cup_{S^3_\Lie \times \bar{K} } \overline{U_{S^3\bar{K}}^\stri}.  
    \end{align}
    This manifold satisfies 
    \begin{align}
        [L_{24}^\spin] = \left\langle [\bar{K}], [S^3_\Lie], 1\right\rangle_{\MSpin/\MString} \in \frac{\pi_{24}\MSpin}{\pi_{24} \MString + [\bar{K}] \cdot \pi_{4}\MSpin} = \frac{\pi_{24}\MSpin}{  \pi_{24}\MString}. 
    \end{align}

    \item A $25$-dimensional string manifold $L^\stri_{25}$ which is given by
    \begin{align}\label{eq_def_L25stri}
        L_{25}^\stri := U_{S^1S^3}^\fr \times \bar{K} \cup_{S^1_\Lie \times S^3_\Lie \times \bar{K} } \overline{S^1_\Lie \times U_{S^3\bar{K}}^\stri}.  
    \end{align}
     This manifold satisfies 
    \begin{align}
        [L_{25}^\stri] = \left\langle [S^1_\Lie], [S^3_\Lie], [\bar{K}]\right\rangle_{\MString/\mathbb{S}} \in \frac{\pi_{25}\MString}{\pi_{5} \mathbb{S} \cdot [\bar{K}] + [S^1_\Lie]\cdot \pi_{24}\MString} = \frac{\pi_{25}\MString}{  [S^1_\Lie] \cdot \pi_{24}\MString}. 
    \end{align}
\end{enumerate}
    
\end{const}

The following is the key property of the above construction. 

\begin{lem}\label{lem_Wit_L24stri}
We can choose the string null bordism $U_{S^3\bar{K}}^\stri$ of $S^3_\Lie \times \bar{K}$ so that we have
    \begin{align}\label{eq_Wit_L24stri}
        \Wit_\stri([L_{24}]^\stri) = 24\Delta. 
    \end{align}
\end{lem}

\begin{proof}
Since the AHR orientation $\AHR_\stri \colon \MString \to \tmf$ is compatible with the formation of Toda brackets, Fact \ref{fact_kappabar} (2) implies that
\begin{align}
    \Wit_\stri(L_{24}^\stri) \equiv 24\Delta \mod 24 \cdot (24\Delta \cdot  \Z + c_4^3 \cdot \Z). 
\end{align}
In order to modify $L_{24}^\stri$ so that the Witten genus is exactly $24\Delta$, notice that the ambiguity of the Toda bracket \eqref{eq_def_L24stri} comes from $24 \cdot \pi_{24}\MString$. By Fact \ref{fact_Wit_image}, 
we see that by replacing the string nullbordism $U^{\stri}_{\bar{\kappa}\nu}$ of $\bar{K} \times S^3_\Lie$ if necessary, we can achieve the desired equality \eqref{eq_Wit_L24stri}. 
\end{proof}

From now on, we use $U_{S^3\bar{K}}^\stri$ satisfying Lemma \ref{lem_Wit_L24stri} to define manifolds $L_{24}^\stri$, $L_{24}^\stri$ and $L_{25}^\stri$ in the above construction. 
The bordism classes of those manifolds are related as follows. 
\begin{claim}\label{claim_relation_L}
    We have
    \begin{align}
        [L_{25}^\stri] &= [S^1_\Lie] \cdot [L_{24}^\spin] \in \pi_{25}\MSpin, \label{eq_relation_L_1}\\
        24[L_{24}^\spin] &= [L_{24}^\stri] \in \pi_{24}\MSpin,\label{eq_relation_L_2} \\
        24[L_{25}^\stri] &= [S^1_\Lie] \cdot [L_{24}^\stri] \in \pi_{25}\MString. \label{eq_relation_L_3}
    \end{align}
\end{claim}
We need to construct the bordisms proving Claim \ref{claim_relation_L} explicitly, as we do below. The equations \eqref{eq_relation_L_1}, \eqref{eq_relation_L_2} and \eqref{eq_relation_L_3} follow from Constructions \ref{const_bordism} (1), \ref{const_bordism_spin_y} (1) and \ref{const_bordism_stri} (1), respectively. 
 
\begin{const}[{spin bordisms $\mathcal{W}^\spin$, $\widetilde{\mathcal{W}}^{\spin} $ and $\mathcal{U}^\spin$}]\label{const_bordism}
    \begin{enumerate}
        \item We define a spin bordism
        \begin{align}
            \mathcal{W}^{\spin} \colon L_{25}^\stri \to S^1_\Lie \times L_{24}^\spin
        \end{align} as follows. 
        By definitions \eqref{eq_def_L25stri} and \eqref{eq_def_L24spin}, it is enough to give a spin-filling of
        \begin{align}
            \left( U_{S^1S^3}^\fr \cup_{S^1_\Lie\times S^3_\Lie} \overline{S_1 \times U^\spin_\nu}\right) \times \bar{K}. 
        \end{align}
        We do this by choosing a spin-filling $U_{\bar{K}}^\spin$ of $\bar{K}$, which is possible because $[\bar{K}] = 0 \in \pi_{20}\MSpin$. 

        \item Using $\mathcal{W}^{\spin}$, we define the spin bordism $\widetilde{\mathcal{W}}^{\spin} \colon (L_{25}^\stri)^{\times 3} \to (S^1_\Lie \times L_{24}^\spin)$ by the following composition of spin bordisms. 
        \begin{align}
            \widetilde{\mathcal{W}}^{\spin} \colon
             (L_{25}^\stri)^{\times 3} &\xrightarrow{\mathcal{W}^\spin \times \id \times \id} (S^1_\Lie \times L_{24}^\spin) \times (L_{25}^\stri)^{\times 2} \\
             &\xrightarrow{\id \times \mathcal{W}^\spin \times \id }(S^1_\Lie \times L_{24}^\spin)^{\times 2} \times (L_{25}^\stri) \\
             &\xrightarrow{\id \times \id \times \mathcal{W}^\spin} (S^1_\Lie \times L_{24}^\spin)^{\times 3}. 
        \end{align}

        \item Recall the equality $[S^1_\Lie]^3 = 12[S^3_\Lie]$ in $\pi_3\MString$. Since we have \eqref{diag_MSpin/MString_deg4}, we can choose a spin nullbordism $U_{(S^1)^{\times 3}}^\spin$ of $(S^1_\Lie)^{\times 3}$ so that
        \begin{align}\label{eq_def_U_eta^3}
            \left[U_{(S^1)^{\times 3}}^\spin, (S^1_\Lie)^{\times 3}\right] = 12[D^4, S^3_\Lie] \in \pi_4\MSpin/\MString. 
        \end{align}
        Using that, we define a spin nullbordism $\mathcal{U}^\spin$ of $(L_{25}^\stri)^{\times 3}$ by the following composition of spin bordisms. 
        \begin{align}
            \mathcal{U}^\spin := (L_{25}^\stri)^{\times 3} \xrightarrow{\widetilde{\mathcal{W}}^\spin} (S^1_\Lie \times L_{24}^\spin) \simeq (S^1_\Lie)^{\times 3} \times (L_{24}^\spin)^{\times 3} 
            \xrightarrow{U_{(S^1)^{\times 3}}^\spin \times (\id)^{\times 3}} \varnothing. 
        \end{align}
    \end{enumerate}
\end{const}

\begin{const}[{spin bordisms $\mathcal{Y}^{\spin}$ and $\widetilde{\mathcal{Y}}^\spin$}]\label{const_bordism_spin_y}
    \begin{enumerate}
        \item We define a spin bordism
    \begin{align}
        \mathcal{Y}^{\spin} \colon 24 L_{24}^\spin \to  L_{24}^\stri
    \end{align}
    as follows. By definitions \eqref{eq_def_L24spin} and \eqref{eq_def_L24stri}, it is enough to give a spin-filling of
    \begin{align}
        \left( 24 U_{S^3}^\spin \cup_{24S^3_\Lie }  \overline{U_{24S^3}^\fr}\right) \times \bar{K}. 
    \end{align}
    We do this by using the same string-filling $U_{\bar{K}}^\spin$ of $\bar{K}$ as the one used in Construction \ref{const_bordism}. 
    \item Using $\mathcal{Y}^\spin$, we define a spin bordism
    \begin{align}
        \widetilde{\mathcal{Y}}^\spin \colon \left(24 L_{24}^\spin\right)^{\times 3} \to \left( L_{24}^\stri \right)^{\times 3}
    \end{align}
    by the same composition as we did to produce $\widetilde{\mathcal{W}}^{\spin}$ out of $\mathcal{W}^{\spin}$ in Construction \ref{const_bordism}. 
    \end{enumerate}
\end{const}

\begin{const}[{string bordisms $\mathcal{X}^{\stri}$ and $\widetilde{\mathcal{X}}^\stri$}]\label{const_bordism_stri}
\begin{enumerate}
    \item We define a string bordism
    \begin{align}
        \mathcal{X}^{\stri} \colon 24 L_{25}^\stri \to S^1_\Lie \times L_{24}^\stri
    \end{align}
    as follows. By definitions \eqref{eq_def_L24stri} and \eqref{eq_def_L25stri}, it is enough to give a string-filling of
    \begin{align}\label{eq_filling_Xstri}
        \left( 24 U_{S^1S^3}^\fr \cup_{24S^3_\Lie \times S^1_\Lie} S^1_\Lie \times U_{24S^3}^\fr\right) \times \bar{K}. 
    \end{align}
    Note that the $5$-dimensional stably framed manifold $ \left( 24 U_{S^1S^3}^\fr \cup_{24S^3_\Lie \times S^1_\Lie} S^1_\Lie \times U_{24S^3}^\fr\right)$ realizes the Toda bracket $\left \langle 24, [S^3_\Lie], [S^1_\Lie]\right\rangle_{\mathbb{S}/\mathbb{S}}$, which is necessarily zero because $\pi_5 \mathbb{S}=0$. 
    Thus we choose its stably framed nullbordism $U_{\langle 24, S^3, S^1 \rangle}^\fr$ and use the induced string-filling of \eqref{eq_filling_Xstri}. 
    
    \item Using $\mathcal{X}^\stri$, we define a string bordism
    \begin{align}
        \widetilde{\mathcal{X}}^\stri \colon \left(24 L_{25}^\stri\right)^{\times 3} \to \left( S^1_\Lie \times L_{24}^\stri \right)^{\times 3}
    \end{align}
    by the same composition as we did to produce $\widetilde{\mathcal{W}}^{\spin}$ out of $\mathcal{W}^{\spin}$ in Construction \ref{const_bordism}. 
\end{enumerate}
   
\end{const}

The bordisms constructed above are related as follows. 
\begin{lem}\label{lem_compare_bordism}
    The $26$-dimensional spin manifold given by the following cyclic composition of bordisms
    \begin{align}
        \xymatrix{
        24L_{25}^\stri \ar[r]^-{24\mathcal{W}^\spin} & 24 S^1_\Lie \times L_{24}^\spin \ar[ld]^-{S^1_\Lie \times \mathcal{Y}^\spin}  \\
        S^1_\Lie \times L_{24}^\stri \ar[u]^-{\overline{\mathcal{X}^\stri}}&
        }
    \end{align}
    is spin-nullbordant. 
    Equivalently, the two bordisms
    \begin{align}\label{eq_compare_bordism}
        \mathcal{X}^\stri \mbox{ and } \left( S^1_\Lie\times \mathcal{Y}^\spin \right)\circ 24\mathcal{W}^\spin \colon 24L_{25}^\stri \to S^1_\Lie \times L_{24}^\stri
    \end{align}
    are bordant as spin bordisms. 
\end{lem}

\begin{proof}
    Recall from the definitions \eqref{eq_def_L24spin}, \eqref{eq_def_L24stri} and \eqref{eq_def_L25stri} that we have the following gluing construction of the manifolds in question, 
    \begin{align}\label{eq_mfds_in_question}
        24L_{25}^\stri &= 24U_{S^1S^3}^\fr \times \bar{K} \cup_{24S^1_\Lie \times S^3_\Lie \times \bar{K} } \overline{24 S^1_\Lie \times U_{S^3\bar{K}}^\stri} \\
        24 S^1_\Lie \times L_{24}^\spin &= 24 S^1_\Lie \times U_{S^3}^\spin \times \bar{K} \cup_{24S^1_\Lie \times S^3_\Lie \times \bar{K} } \overline{24 S^1_\Lie \times U_{S^3\bar{K}}^\stri} \\
        S^1_\Lie \times L_{24}^\stri &= S^1_\Lie\times U_{24S^3}^\fr \times \bar{K} \cup_{24 S^1_\Lie \times S^3_\Lie \times \bar{K}} \overline{24 S^1_\Lie \times U^\stri_{S^3 \bar{K}}}. 
    \end{align}
    Also recall from Constructions \ref{const_bordism}, \ref{const_bordism_spin_y} and \ref{const_bordism_stri} that each bordism $24\mathcal{W}^\spin$, $\left( S^1_\Lie\times \mathcal{Y}^\spin \right)$ and $\mathcal{X}^\stri$ fixes the second component $\overline{24 S^1_\Lie \times U_{S^3\bar{K}}^\stri}$ of the manifolds in \eqref{eq_mfds_in_question}. 
    Indeed, the bordisms are constructed by taking nullbordisms of manifolds formed by gluing the first factor of each two of the manifolds in \eqref{eq_mfds_in_question}. 

    We compare two spin bordisms \eqref{eq_compare_bordism}, which are identified with two spin-fillings of 
    \begin{align}
         \left( 24 U_{S^1S^3}^\fr \cup_{24S^3_\Lie \times S^1_\Lie} S^1_\Lie \times U_{24S^3}^\fr\right) \times \bar{K}. 
    \end{align}
    $\mathcal{X}^\stri$ corresponds to the string filling
    \begin{align}
        U_{\langle 24, S^3, S^1 \rangle}^\fr \times \bar{K}, 
    \end{align}
    whereas $\left( S^1_\Lie\times \mathcal{Y}^\spin \right)\circ 24\mathcal{W}^\spin$ corresponds to the spin filling
    \begin{align}
        \left( 24 U_{S^1S^3}^\fr \cup_{24S^3_\Lie \times S^1_\Lie} S^1_\Lie \times U_{24S^3}^\fr\right) \times U_{\bar{K}}^\spin. 
    \end{align}
    Thus, the spin manifold with corners
    \begin{align}
        U_{\langle 24, S^3, S^1 \rangle}^\fr \times U_{\bar{K}}^\spin
    \end{align}
    gives a spin bordism between the two bordisms \eqref{eq_compare_bordism} as desired. This finishes the proof of Lemma \ref{lem_compare_bordism}. 
\end{proof}

The following lemmas are the key in the proof of Proposition \ref{prop_lowerbound_12}.

\begin{lem}\label{lem_76_relative_mfd}
    We have the following equality in $\pi_{76}\MSpin/\MString$. 
    \begin{align}\label{eq_lem_76_relative_mfd}
        24^3 \cdot  \left[\mathcal{U}^\spin, (L_{25}^\stri)^{\times 3}\right] = \left[L_{24}^\stri\right]^3 \cdot \left[U_{(S^1)^{\times 3}}^\spin, (S^1_\Lie )^{\times 3}\right] . 
    \end{align}
\end{lem}

\begin{proof}
Recall that equalities in relative bordism groups are given by manifolds with corners, more precisely $\langle 2\rangle$-manifolds in the sense of \cite{Janich1968}. 
The data of $\langle 2 \rangle$-manifold $M$ consists of a decomposition $\del M= \del_0 M \cup \del_1 M$ of its boundary so that $\del_{01}M := \del_0 M \cap \del_1 M $ is the corner of $M$. 
In order to prove \eqref{eq_lem_76_relative_mfd}, we need to construct a $77$-dimensional $\langle 2 \rangle$-manifold $\mathcal{V}$ so that
\begin{itemize}
    \item $\mathcal{V}$ is equipped with a spin structure and on $\del_0 \mathcal{V}$ it is lifted to a string structure, 
    \item We have an identification of relative spin/string manifolds, 
    \begin{align}
        \left(\del_1 \mathcal{V}, \del_{01}\mathcal{V}\right)
        \simeq 24^3\left(\mathcal{U}^\spin, (L_{25}^\stri)^{\times 3}\right) \sqcup \overline{ (L_{24}^\stri)^3 \times \left(U_{(S^1)^{\times 3}}^\spin, (S^1_\Lie )^{\times 3}\right)}. 
    \end{align}
    In particular, $\del_{01}\mathcal{V}$ is a string bordism from $\left(24 L_{24}^\stri\right)^{\times 3}$ to $\left( L_{24}^\stri \times S^1_\Lie \right)^{\times 3}$. 
\end{itemize}
We construct such $\mathcal{V}$ as follows. First recall that $\mathcal{U}^\spin$ is given by the following composition of spin bordisms, 
\begin{align}
    \mathcal{U}^\spin := (L_{25}^\stri)^{\times 3} \xrightarrow{\widetilde{\mathcal{W}}^\spin} (S^1_\Lie \times L_{24}^\spin) \simeq (S^1_\Lie)^{\times 3} \times (L_{24}^\spin)^{\times 3} 
            \xrightarrow{U_{(S^1)^{\times 3}}^\spin \times (\id)^{\times 3}} \varnothing.
\end{align}

Accordingly, we define
\begin{align}\label{eq_proof_76_relative_mfd}
    \mathcal{V} = \mathcal{V}_L \cup_{(S^1_\Lie)^{\times 3} \times \widetilde{\mathcal{Y}}^\spin} \mathcal{V}_R, 
\end{align}
where $\mathcal{V}_L$ and $\mathcal{V}_R$ are spin $\langle 2\rangle$-manifolds as follows. 
\begin{itemize}
    \item $\mathcal{V}_L$ is the $\langle 2 \rangle$-manifold realizing the bordism between two spin bordisms
    \begin{align}
        \widetilde{\mathcal{X}}^\stri \mbox{ and } \left( S^1_\Lie\times \widetilde{\mathcal{Y}}^\spin \right)\circ 24^3 \widetilde{\mathcal{W}}^\spin \colon \left(24L_{25}^\stri\right)^{\times 3} \to \left(S^1_\Lie \times L_{24}^\stri \right)^{\times 3}, 
    \end{align}
    whose existence is guaranteed by Lemma \ref{lem_compare_bordism}. We have
    \begin{align}
        \del_0 \mathcal{V}_L &\simeq \widetilde{\mathcal{X}}^{\stri} \sqcup \overline{(S^1_\Lie)^{\times 3} \times \widetilde{\mathcal{Y}}^\spin}, \\
        \del_1 \mathcal{V}_L & \simeq \widetilde{\mathcal{W}}^\spin \sqcup
        \overline{\left( L_{24}^\stri \times S^1_\Lie \right)^{\times 3} \times [0, 1]}. 
    \end{align}

    \item $\mathcal{V}_R$ is defined by
    \begin{align}
        \mathcal{V}_R := U_{(S^1)^{\times 3}}^\spin \times \widetilde{\mathcal{Y}}^\spin. 
    \end{align}
    We have
    \begin{align}
        \del_0 \mathcal{V}_R &\simeq (S^1_\Lie)^{\times 3} \times \widetilde{\mathcal{Y}}^\spin, \\
        \del_1 \mathcal{V}_R &\simeq 24^3 U_{(S^1)^{\times 3}}^\spin \times (L_{24}^\spin)^{\times 3} \sqcup \overline{U_{(S^1)^{\times 3}}^\spin \times(L_{24}^\stri)^{\times 3}}. 
    \end{align}
\end{itemize}
We see that $\mathcal{V}$ defined by \eqref{eq_lem_76_relative_mfd} has the desired property. This completes the proof of Lemma \ref{lem_76_relative_mfd}. 
\end{proof}

\begin{lem}\label{lem_rel_Wit_76}
    We have
    \begin{align}
        \Wit_{\rel}\left(\left[\mathcal{U}^\spin, (L_{25}^\stri)^{\times 3}\right]\right)
        = E_2 \cdot \Delta^3 \mod \MF_{36}^\Q, 
    \end{align}
            where $E_2(q)$ is the second Eisenstein series with the normalization so that $E_2(q) = 1 + O(q)$. 
\end{lem}
\begin{proof}
This follows from the following computation. 
\begin{align}
    \Wit_{\rel}\left(\left[\mathcal{U}^\spin, (L_{25}^\stri)^{\times 3}\right]\right) 
    &= \frac{1}{24^3} \Wit_{\rel}\left( [L_{24}^\stri]^3 \cdot [U_{(S^1)^{\times 3}}^\spin, (S^1_\Lie )^{\times 3}]\right) \notag \\
    &= \frac{1}{24^3} \Wit([L_{24}^\stri])^3 \cdot \Wit_{\rel}\left([U_{(S^1)^{\times 3}}^\spin, (S^1_\Lie )^{\times 3}]\right) \notag \\
    &= \frac{1}{24^3} \Wit([L_{24}^\stri])^3 \cdot 12 \Wit_{\rel}([D^4, S^3_\Lie]) \notag \\
    &= \frac{1}{24^3} \cdot (24\Delta)^3 \cdot 12 \cdot \frac{E_2}{12} \pmod{\MF_{36}^\Q} \notag \\
    &=E_2 \cdot \Delta^3 \pmod{\MF_{36}^\Q} \notag
\end{align}
Here, the first equality follows from Lemma \ref{lem_76_relative_mfd}, the second follows from the product formula in Lemma \ref{lem_product_Wit}, the third follows from \eqref{eq_def_U_eta^3}, the fourth follows from Lemma \ref{lem_Ind_rel_D4S3}.  
\end{proof}

\begin{proof}[Proof of Proposition \ref{prop_lowerbound_12}]
We use the pairing
    \begin{align}
        \langle-,-\rangle_\SQFT: \pi_{-12 \times 24}\SQFT \times \pi_{76 + 8 \times 24}\MSpin/\MString\to \bZ.
    \end{align}
We use the class $[\mathcal{U}^\spin, (L_{25}^\stri)^{\times 3}] \in \pi_{76}\MSpin/\MString$ in Construction \ref{const_bordism}. 
Moreover, by Fact \ref{fact_Wit_image}, we have a $15 \times 24$-dimensional string manifold $M_{8 \times 24}^\stri$ with $\Wit_\stri([M_{8 \times 24}^\stri]) = 3\Delta^{8}$. 
Thus we have a class
\begin{align}\label{eq_rel_mfd_total}
    [M_{8 \times 24}^\stri] \cdot \left[\mathcal{U}^\spin, (L_{25}^\stri)^{\times 3}\right]\in \pi_{8 \times 24 + 76}\MSpin/\MString. 
\end{align}
Suppose there exists $\mathcal{T} \in \pi_{-12 \times 24} \SQFT$ with $\Phi(\mathcal{T}) = k \Delta^{-12}$. 
    Then, by the formula \eqref{eq_formula_pariring} and the integrality of the pairing, together with Lemma \ref{lem_rel_Wit_76} and Lemma \ref{lem_product_Wit}, we get
    \begin{align}
        \Z &\ni \left\langle \mathcal{T}, [M_{15 \times 24}^\stri] \cdot \left[\mathcal{U}^\spin, (L_{25}^\stri)^{\times 3}\right] \right\rangle_\SQFT = \left.\frac{1}{2}\Delta \cdot k \Delta^{-12} \cdot\left(1 + O(q) \right) \cdot \Delta^3 \cdot 3\Delta^{8} \right|_{q^0} 
        = \frac{3k}{2}. 
    \end{align}
    This gives the desired result.

\end{proof}

\appendix

\section{Toda brackets via manifolds}
\label{app:Toda}
For details on Toda bracket, see the original article \cite{Toda}. 
The Toda bracket we use is the following. Let $R$ be an $E_\infty$ ring spectrum and $M$ be an $R$-module spectrum. 
Let $x, y \in \pi_* R$ and $z \in \pi_* M$. 
Assume that we have $xy = 0 \in \pi_* R$ and $yz = 0 \in \pi_* M$. 
In this setting, we get the {\it Toda bracket} 
\begin{align}\label{eq_Toda}
    \langle x, y, z \rangle_{M/R} \in \frac{\pi_{|x|+ |y| + |z| + 1}M}{\left( \pi_{|x| + |y| + 1} R \cdot z\right) + \left(x \cdot \pi_{|y| + |z| + 1}M\right)}. 
\end{align}
as follows. 

\begin{defn}[Toda brackets, \cite{Toda}]\label{def_Toda_htpy}
In the above settings,
consider the following diagram of spectra, 

\begin{align}\label{diag_Toda}
    \xymatrix{
&  & & \Sigma^{|x| + |y| + |z| + 1} \mathbb{S} \ar[d]^-{x\cdot } \ar@{.>}[ld]^-{u_{xy}} & \\
    \Sigma^{|y| + |z|} R \ar[r]^-{y \cdot}  & \Sigma^{|z|} R \ar[r] \ar[d]^-{z \cdot} & \Sigma^{|z|} Cy \ar[r] \ar@{.>}[ld]^-{u_{yz}} & \Sigma^{|y|+ |z| +1} R \ar[r]^-{y\cdot } & \Sigma^{|z|+1} R \\
    & M & & & 
    }, 
\end{align}
where $Cy$ is the cofiber of $y\cdot \colon \Sigma^{|y|}R \to R$, so that the middle row is a fiber sequence. 
Since we have $xy=0 \in \pi_*R$, we can choose a lift $u_{xy}$ as indicated in the diagram. The ambiguity of the lift is parametrized by $\pi_{|x|+|y| + 1} R$. 
Similarly, since we have $yz=0 \in \pi_{*}M$, we can choose a lift $u_{yz}$ as indicated in the diagram. The ambiguity of the lift is parametrized by $\pi_{|y| + |z| + 1}M$. 
Then we define the Toda bracket to be the following composition, 
\begin{align}
    \langle x, y, z \rangle_{M/R} := u_{yz} \circ u_{xy} \colon \Sigma^{|x| + |y| + |z| + 1} \mathbb{S} \to M. 
\end{align}
Combining the ambiguity in the choices of $u_{xy}$ and $u_{yz}$ explained above, we see that the Toda bracket is well-defined as an element in the right hand side of \eqref{eq_Toda}. 

\end{defn}
 
We are interested in the situation where $R$ and $M$ are Thom spectra of the form $R=M\mathcal{S}$ and $M=M\mathcal{S}'$, where $\mathcal{S}$ is a multiplicative structure and $\mathcal{S}'$ is an $\mathcal{S}$-module structure. 
The example appearing in this article is $(M\mathcal{S}, M\mathcal{S}') = (\MString, \MSpin)$ and $(\mathbb{S}, \MString)$. 
In this settings, we have the following geometric realization of the Toda bracket \eqref{eq_Toda}. 

\begin{defn}[Manifold-level construction of Toda brackets]\label{def_Toda_mfd}
Let us choose closed $\mathcal{S}$ manifolds $X$ and $Y$ and a closed $\mathcal{S}'$-manifold $Z$ with $x = [X]$, $y=[Y]$ and $z = [Z]$. 
Then we can choose a $\mathcal{S}$-nullbordism $U_{xy}$ of $X \times Y$ and a $\mathcal{S}'$-nullbordism $U_{yz}$ of $Y \times Z$. 
Then let us form 
\begin{align}
    W := U_{xy} \times Z \cup_{X \times Y \times Z} \overline{X \times U_{yz}}. 
\end{align}
Then $W$ is naturally a closed $\mathcal{S}'$-manifold with dimension $(|x| + |y| + |z| + 1)$. 
The ambiguity of the choices of nullbordisms $U_{xy}$ and $U_{yz}$ affects the resulting homotopy class $[W] \in \pi_{|x|+ |y| + |z| + 1} M\mathcal{S}'$ by exactly by $\left( \pi_{|x| + |y| + 1} M\mathcal{S}  \cdot z\right)$ and $\left( x \cdot \pi_{|y| + |z| + 1}M\mathcal{S}'\right)$, respectively. 
Thus the following element is well-defined, which we call the {\it manifold-level Toda bracket}, 
\begin{align}
    \langle x, y, z \rangle_{M\mathcal{S}'/M\mathcal{S}}^{\mathrm{mfd}} := [W] \in \frac{\pi_{|x|+ |y| + |z| + 1}M}{\left( \pi_{|x| + |y| + 1} R \cdot z\right) + \left(x \cdot \pi_{|y| + |z| + 1}M\right)}. 
\end{align}
\end{defn}

The above manifold-level construction of Toda bracket (Definition \ref{def_Toda_mfd}) indeed coincides the homotopy-theoretic definition of Toda bracket (Definition \ref{def_Toda_htpy}), as follows. 

\begin{prop}\label{prop_Toda_htpy=mfd}
    In the above settings, we have
    \begin{align}
        \langle x, y, z \rangle_{M\mathcal{S}'/M\mathcal{S}}= \langle x, y, z \rangle_{M\mathcal{S}'/M\mathcal{S}}^{\mathrm{mfd}}, 
    \end{align}
    where the left and right hand sides are given in Definition \ref{def_Toda_htpy} and Definition \ref{def_Toda_mfd}, respectively. 
\end{prop}
\begin{proof}
    This directly follows from the Pontryagin-Thom construction. For details of Pontryagin-Thom construction in the presence of boundaries, see \cite{LauresANSS}. 
\end{proof}



\providecommand{\bysame}{\leavevmode\hbox to3em{\hrulefill}\thinspace}
\providecommand{\MR}{\relax\ifhmode\unskip\space\fi MR }
\providecommand{\MRhref}[2]{%
  \href{http://www.ams.org/mathscinet-getitem?mr=#1}{#2}
}
\providecommand{\href}[2]{#2}
\providecommand{\doihref}[2]{\href{#1}{#2}}
\providecommand{\arxivfont}{\tt}

\end{document}